\documentclass[12pt]{article}

\usepackage{mathrsfs}
\usepackage{amsfonts}
\usepackage{amssymb,amsmath,amsbsy,amsthm}
\usepackage{blkarray}
\usepackage[dvips]{graphics}
\usepackage{ae}
\usepackage{bm}
\usepackage{graphicx}
\usepackage{cite}
 \usepackage{indentfirst}
\usepackage{lineno}
\usepackage{titlesec}
\usepackage{enumitem}
\usepackage{color}
\usepackage{subfigure}
\usepackage{caption}
\usepackage{aliascnt}
\usepackage{varioref}
\usepackage{hyperref}
\usepackage{chngcntr}

\hypersetup{colorlinks=true,linkcolor=cyan,anchorcolor=cyan,citecolor=red,CJKbookmarks=True,
,urlcolor=blue}
\usepackage{cleveref}

\def\rank{\mbox{\rm rank}}
\def\NBC{{\rm NBC}}

\def\span{\mbox{\rm span}}

\def\min{\mbox{\rm min}}
\def\max{\mbox{\rm max}}

\def\And{\mbox{\rm ~and~}}

\def\If{\mbox{\rm ~if~}}
\def\For{\mbox{\rm ~for~}}

\def\i{\mbox{\rm (\hspace{0.2mm}i\hspace{0.2mm})}\,}
\def\ii{\mbox{\rm (\hspace{-0.1mm}i\hspace{-0.2mm}i\hspace{-0.1mm})}\,}

\def\Sum{\textstyle\sum\limits}

\def\({\mbox{\rm (}}\def\){\mbox{\rm )}}

\makeatletter

\newcommand{\Rmnum}[1]{\expandafter\@slowromancap\romannumeral #1@}

\makeatother
\newtheorem{theorem}{Theorem}[section]
\newaliascnt{lemma}{theorem}
\newtheorem{lemma}[lemma]{Lemma}
\aliascntresetthe{lemma}

\newaliascnt{proposition}{theorem}

\aliascntresetthe{proposition}

\newaliascnt{fact}{theorem}

\aliascntresetthe{fact}

\newaliascnt{definition}{theorem}

\aliascntresetthe{definition}

\newaliascnt{conjecture}{theorem}

\aliascntresetthe{conjecture}

\newaliascnt{corollary}{theorem}
\newtheorem{corollary}[corollary]{Corollary}
\aliascntresetthe{corollary}

\newaliascnt{claim}{theorem}

\aliascntresetthe{claim}

\newaliascnt{problem}{theorem}

\aliascntresetthe{problem}

\newaliascnt{remark}{theorem}

\aliascntresetthe{remark}

\newaliascnt{example}{theorem}
\newtheorem{example}[example]{Example}
\aliascntresetthe{example}

\begin{document}
\normalsize
\title{{\bf One-element Extensions of Hyperplane Arrangements}}
\author{Hang Cai\\
\small School of Mathematics, Hunan University, Changsha, China\\
\small caihang@hnu.edu.cn
\and
Houshan Fu\thanks{Corresponding author}\\
\small School of Mathematics and Information Science\\
\small Guangzhou University, Guangzhou, China\\
\small fuhoushan@gzhu.edu.cn
\and
Suijie Wang\thanks{Supported by 12171487}\\
\small School of Mathematics\\
\small Hunan University, Changsha, China\\
\small  wangsuijie@hnu.edu.cn}
\date{}
\maketitle
\captionsetup[figure]{labelformat=empty}
\begin{abstract}
We classify one-element extensions of a hyperplane arrangement by the induced adjoint arrangement. Based on the classification, several kinds of combinatorial invariants including Whitney polynomials, characteristic polynomials, Whitney numbers and face numbers, are constants on those strata associated with the induced adjoint arrangement, and also order-preserving with respect to the intersection lattice of the induced adjoint arrangement.
As a byproduct, we obtain a convolution formula on the characteristic polynomials $\chi(\mathcal{A}+H_{\bm\alpha,a},t)$ when $\mathcal{A}$ is defined over a finite field $\mathbb{F}_q$ or a rational arrangement.
\vspace{1ex} \vspace{1ex}\\
MSC classes: 52C35.
\end{abstract}
\section{Introduction}\label{SEC1}
\paragraph{1.1. Main purposes and background.}
The starting point of this paper is coming from Fu-Wang's work in \cite{Fu-Wang2021}, which has characterized the classification of the linear one-element extensions of a linear arrangement by the adjoint arrangement. We shall extend their work to general hyperplane arrangement. Our main goal is to classify intersection semi-lattices of one-element extensions of a hyperplane arrangement,  and the related combinatorial invariants involving in Whitney polynomials, characteristic polynomials, Whitney numbers and face numbers. According to the classification, we further obtain that these combinatorial invariants are order-preserving with respect to the intersection lattice of the induced adjoint arrangement, and a convolution formula on the characteristic polynomials.

For this purpose, we shall fix our notation and introduce the definition of one-element extension of hyperplane arrangement. Throughout this paper, we use $[m]$, $\mathbb{Z}$, $\mathbb{F}_q$, $\mathbb{Q}$, $\mathbb{R}$ and $\mathbb{F}$ to denote the set of $\{1,2,\ldots,m\}$, the set of integers,  the finite field with $q$ elements, the field of rational numbers, the field of real numbers and a general field respectively. For any vectors $\bm \alpha\in\mathbb{F}^d, (\bm\alpha, a)\in\mathbb{F}^{d+1}$, the notations $H_{\bm\alpha}$ and $H_{\bm\alpha,a}$ always refer to the linear hyperplane $H_{\bm\alpha}:\bm\alpha\cdot\bm x=0$ and  the affine hyperplane $H_{\bm\alpha,a}:\bm\alpha\cdot\bm x=a$ in $\mathbb{F}^d$ respectively. In particular, $H_{\bm\alpha}=H_{\bm\alpha,0}$. Let $\mathcal{A}$ be a hyperplane arrangement in the $d$-dimensional vector space $\mathbb{F}^d$. The hyperplane arrangement
\[
\mathcal{A}+ H_{\bm\alpha,a}:=\mathcal{A}\cup\{ H_{\bm\alpha,a}\}
\]
in $\mathbb{F}^d$ is called a {\em one-element extension} of $\mathcal{A}$. In particular, the hyperplane arrangement $\mathcal{A}+H_{\bm\alpha}$ is called a {\em linear one-element extension} of $\mathcal{A}$. Note that all normal vectors of hyperplanes in $\mathcal{A}$ naturally define an $\mathbb{F}$-vector matroid $\mathcal{M}(\mathcal{A})$ (a matroid represented by vectors over $\mathbb{F}$). In matroid language, when $\mathcal{A}$ is a linear arrangement, a linear one-element extension of $\mathcal{A}$ is indeed an $\mathbb{F}$-representable extension of the $\mathbb{F}$-vector matroid $\mathcal{M}(\mathcal{A})$, which can be viewed as a kind of single-element extensions of a matroid. It was first introduced by Crapo \cite{Crapo1965} in 1965, which gives a characterization of all single-element extensions by the linear subclasses of a matroid.

\paragraph{1.2. Outline} The remainder of this paper is organized as follows. After collecting the necessary concepts in Subsection 2.1 and Subsection 2.2, we shall state our main results in Subsection 2.3. \autoref{SEC2} proves \autoref{classification-1} and further investigates the classifications of  several special types of one-element extensions, see \autoref{classification-2}, \autoref{classification-3} and \autoref{classification-4}. In \autoref{SEC3}, we give a proof of \autoref{comparison-1} and also establish the order-preserving relations of some combinatorial invariants associated with the several special types of one-element extensions, see \autoref{comparison-2}, \autoref{comparison-3} and \autoref{comparison-4}. \autoref{SEC4-0} shows \autoref{Convolution-Formula}. As an application of one-element extension of hyperplane arrangement, in \autoref{SEC4} we study the restrictions of a hyperplane arrangement, and give the classification of their intersection semi-lattices in \autoref{contraction} and the order-preserving relations of their Whitney numbers of both kinds and region numbers in \autoref{restriction1}.

\section{Preliminaries}\label{SEC2-0}
\paragraph{2.1. Basic concepts.}
Our terminology and notations on hyperplane arrangement refer to \cite{Stanley2007}. A {\em hyperplane arrangement} $\mathcal{A}$ is a finite collection of hyperplanes in a $d$-dimensional vector space $V$. In particular, $\mathcal{A}$ is called a  {\em linear arrangement} if all hyperplanes in $\mathcal{A}$ pass through the origin. The {\em intersection semi-lattice} $L(\mathcal{A})$  consists of all nonempty intersections of some hyperplanes in $\mathcal{A}$, ordered by the reverse inclusion and including $V=\bigcap_{H\in\emptyset}H$ as the minimal element. For $X\in L(\mathcal{A})$, the {\em localization} $\mathcal{A}_X$ of $\mathcal{A}$ at $X$ is defined as
\[
\mathcal{A}_X:=\{H\in\mathcal{A}\mid X\subseteq H\},
\]
and the {\em restriction} $\mathcal{A}/X$ is a hyperplane arrangement in $X$ given by
\[
\mathcal{A}/X:=\{H\cap X\ne\emptyset\mid  H\in \mathcal{A}\setminus\mathcal{A}_X\}.
\]
The {\em complement} of $\mathcal{A}$ is $M(\mathcal{A})=V\setminus\bigcup_{H\in \mathcal{A}}H$. It is clear that the whole space $V$ has the set partition
\begin{equation}\label{complement-partition}
V=\bigsqcup_{X\in L(\mathcal{A})}M(\mathcal{A}/X).
\end{equation}
Namely, for each point ${\bm x}\in V$, there is a unique element $X\in L(\mathcal{A})$ satisfying ${\bm x}\in M(\mathcal{A}/X)$, denoted by $X_{\bm x}$. Indeed, $X_{\bm x}$ is the inclusion-minimal element of $L(\mathcal{A})$ containing ${\bm x}$.

The {\em characteristic polynomial} $\chi(\mathcal{A},t)$ of $\mathcal{A}$ is
\begin{equation*}\label{Characteristic-Polynomial}
\chi(\mathcal{A},t):=\sum\limits_{X\in L(\mathcal{A})}\mu(V,X)t^{{\rm dim}(X)}= w_0^+(\mathcal{A})t^d- w_1^+(\mathcal{A})t^{d-1}+\cdots+(-1)^d w_d^+(\mathcal{A}),
\end{equation*}
where $\mu$ is the M\"obius function of $L(\mathcal{A})$. In particular, we assume $\chi(\mathcal{A},t)=0$ if the ambient space $V$ is a member of $\mathcal{A}$. Each unsigned coefficient $ w_k^+(\mathcal{A})$ is called the  {\em $k$-th unsigned Whitney number of the first kind}.  The number  $ W_k(\mathcal{A})$ of  members in $L(\mathcal{A})$ of codimension $k$ is called the {\em $k$-th Whitney number of the second kind}. As a generalization of characteristic polynomial, Zaslavsky \cite{Zaslavsky1975} introduced the {\em Whitney polynomial} $ w(\mathcal{A};s,t)$ of $\mathcal{A}$ as
\begin{equation}\label{Whitney-Polynomial}
 w(\mathcal{A};s,t):=\sum\limits_{X\le Y {\rm\,in\,} L(\mathcal{A})}\mu(X,Y)s^{d-{\rm dim}(X)}t^{{\rm dim}(Y)}=\sum_{0\le j\le i\le d}(-1)^jc_{  ij}(\mathcal{A})s^{d-i}t^{i-j}.
\end{equation}
Likewise, we assume $ w(\mathcal{A};s,t)=0$ if $V\in\mathcal{A}$. Clearly it can be written as
\begin{equation}\label{Whitney-Polynomial-I}
 w(\mathcal{A};s,t)=\sum_{  X\in L(\mathcal{A})}\chi(\mathcal{A}/X,t)\cdot s^{d-{\rm dim}(X)}.
\end{equation}
Indeed,  $\chi(\mathcal{A},t)=w(\mathcal{A};0,t)$.  The {\em doubly indexed Whitney number of the first kind} was proposed by Zaslavsky \cite{Zaslavsky1981} and defined as
\[
w_{ij}(\mathcal{A}):=\sum_{X\le Y {\rm\,in\,} L(\mathcal{A})}\{\mu(X,Y):{\rm codim}(X)=i,\;{\rm codim}(Y)=j\}.
\]
More information about it,  see \cite{Greene-Zaslavsky1983,Zaslavsky2002}. Actually, each $c_{ij}(\mathcal{A})$ equals $(-1)^jw_{d-i,d+j-i}(\mathcal{A})$.

When $V$ is a Euclidean space,  $M(\mathcal{A})$ consists of finitely many connected components, called {\em regions} of $\mathcal{A}$. Denote by $r(\mathcal{A})$ the number of regions of $\mathcal{A}$. The celebrated Zaslavsky's formula \cite{Zaslavsky1975} states $r(\mathcal{A})=(-1)^d\chi(\mathcal{A},-1)$. Let $X\in L(\mathcal{A})$ be an affine subspace of dimension $k$. We call each region of $\mathcal{A}/X$ a {\em $k$-dimensional face of $\mathcal{A}$}. Denote by $f_k(\mathcal{A})$ the number of $k$-dimensional faces of $\mathcal{A}$. Then Zaslavsky's formula directly yields the consequence
\[
f_k(\mathcal{A})=\sum_{X\in L(\mathcal{A}),\,\dim(X)=k}(-1)^{k}\chi(\mathcal{A}/X,-1).
\]
In particular, $f_d(\mathcal{A})=r(\mathcal{A})$.

\paragraph{2.2. Induced adjoint arrangement.} To obtain our results, we shall introduce adjoint arrangement and induced adjoint arrangement. Unless otherwise stated, we always regard $\mathcal{A}$  as the hyperplane arrangement $\mathcal{A}:=\big\{H_{\bm \alpha_i,a_i}\mid i\in[m]\big\}$ in $\mathbb{F}^d$ later. The {\em linearization} of $\mathcal{A}$  is a linear arrangement $\mathcal{A}^o$ in $\mathbb{F}^d$ defined to be
\[
\mathcal{A}^o:=\big\{H_{\bm \alpha_i}\mid i\in[m]\big\}.
\]
When $\mathcal{A}$ is essential, i.e., the whole space $\mathbb{F}^d$ is exactly spanned by the normal vectors ${\bm\alpha}_1,{\bm\alpha}_2,\ldots,{\bm\alpha}_m$, then $\mathcal{A}^o$ is also essential. In this case, we always assume that
\begin{eqnarray*}
L_{0}(\mathcal{A}):=\{{\bm v}_1,{\bm v}_2,\ldots,{\bm v}_{n_0}\}\quad\And\quad
L_1(\mathcal{A}^o):=\{\mathbb{F}{\bm u}_1,\mathbb{F}{\bm u}_2,\ldots,\mathbb{F}{\bm u}_{n_1}\}
\end{eqnarray*}
consist of $0$-dimensional members of $L(\mathcal{A})$ and $1$-dimensional members of $L(\mathcal{A}^o)$ later respectively (unless otherwise stated), where each $\mathbb{F}{\bm u}_i$ denotes the subspace spanned by ${\bm u}_i$. The {\em adjoint arrangement} $\sigma\mathcal{A}^o$ of $\mathcal{A}^o$ is a linear arrangement in $\mathbb{F}^d$ defined as
\begin{equation}\label{adjoint-arrangement}
\sigma\mathcal{A}^o:=\big\{H_{\bm u_i}:\bm u_i\cdot\bm x=0\mid i\in[n_1]\big\},
\end{equation}
which was first constructed by Bixby-Coullard \cite{Bixby-Coullard1988} for vector matroids in 1988. Associated with $L_0(\mathcal{A})$ and $L_1(\mathcal{A}^o)$, the {\em induced adjoint arrangement} $\tilde{\mathcal{A}}$ in $\mathbb{F}^{d+1}$  is defined as
\begin{equation}\label{extension-adjoint-arrangement}
\tilde{\mathcal{A}}:=\big\{H_{\bm\tilde{\bm v}_i}:\tilde{\bm v}_i\cdot{\bm x}=0\mid i\in[n_0]\big\}\cup\big\{H_{\tilde{\bm u}_i}:\tilde{\bm u}_i\cdot{\bm x}=0\mid i\in[n_1]\big\},
\end{equation}
where $\tilde{{\bm v}}_i=({\bm v}_i,-1)$ and $\tilde{{\bm u}}_i=({\bm u}_i,0)$.
For convenience, let
\begin{equation}\label{A_01}
\tilde{\mathcal{A}}_0:=\big\{H_{\tilde{\bm v}_i}:\tilde{\bm v}_i\cdot{\bm x}=0\mid i\in[n_{0}]\big\}\quad\And \quad\tilde{\mathcal{A}}_1:=\big\{H_{\tilde{\bm u}_i}:\tilde{\bm u}_i\cdot{\bm x}=0\mid i\in[n_{1}]\big\}.
\end{equation}

\paragraph{2.3. Main results.}
Suppose $\mathcal{A}$ is essential. For any $(\bm\alpha, a)\in\mathbb{F}^{d+1}$, the  one-element extension $\mathcal{A}+H_{\bm\alpha,a}$ is obtained from $\mathcal{A}$ by adding a hyperplane $H_{\bm\alpha,a}$ to $\mathcal{A}$. Note from \eqref{complement-partition} that $\tilde{\mathcal{A}}$ makes a partition $\mathbb{F}^{d+1}=\bigsqcup_{X\in L(\tilde{\mathcal{A}})}M(\tilde{\mathcal{A}}/X)$, which will classify the intersection semi-lattices $L(\mathcal{A}+H_{\bm\alpha,a})$ for all $(\bm\alpha,a)\in\mathbb{F}^{d+1}$. In particular, $L\big(\mathcal{A}+ H_{\bm0,a}\big)=L(\mathcal{A})$ for $a\in\mathbb{F}\setminus\{0\}$.
\begin{theorem}\label{classification-1}
Suppose $\mathcal{A}$ is essential. Given $X\in L(\tilde{\mathcal{A}})$. For any vectors $({\bm\alpha},a),({\bm\alpha}',a')\in M(\tilde{\mathcal{A}}/X)$, we have
\[ L\big(\mathcal{A}+ H_{\bm\alpha,a}\big)\cong L\big(\mathcal{A}+ H_{\bm\alpha',a'}\big).
\]
\end{theorem}
In fact, when $\mathcal{A}$ is not essential, the classification of the one-element extensions can be reduced to  classifying $\mathcal{A}+H_{\bm\alpha,a}$ for $(\bm\alpha,a)\in O\times\mathbb{F}$ ($O$ is the space spanned by the normal vectors $\bm\alpha_1,\ldots,\bm\alpha_m$ of hyperplanes in $\mathcal{A}$), which will be explained at the end of Subsection 3.1.

According to \autoref{classification-1}, each invariant (unsigned Whitney number of the first kind, Whitney number of the second kind, face number, etc.) of every one-element extension of a hyperplane arrangement can be regarded as a function on the intersection lattice of the induced adjoint arrangement. Next we will further establish the order-preserving relations of these combinatorial invariants with respect to $L(\tilde{\mathcal{A}})$, including the unsigned coefficients $c_{  ij}(\mathcal{A})$ of the  Whitney polynomial, the unsigned Whitney numbers $w_i^+(\mathcal{A})$ of the first kind, the Whitney numbers $ W_i(\mathcal{A})$ of the second kind, the face numbers $f_i(\mathcal{A})$ and the region number $r(\mathcal{A})$. For convenience, we use an united notation  ${\rm INV}(\mathcal{A})$ to denote any invariant among $c_{  ij}(\mathcal{A}), w_i^+(\mathcal{A}), W_i(\mathcal{A}),f_i(\mathcal{A})$, $r(\mathcal{A})$ for all $0\le j\le i\le d$. For example, if $\mathcal{A}$ and $\mathcal{A}'$ are two hyperplane arrangements in $\mathbb{F}^d$, the notation ${\rm INV}(\mathcal{A})\le {\rm INV}(\mathcal{A}')$ means $c_{  ij}(\mathcal{A})\le c_{  ij}(\mathcal{A}'), w_i^+(\mathcal{A})\le w_i^+(\mathcal{A}'), W_i(\mathcal{A})\le W_i(\mathcal{A}'),f_i(\mathcal{A})\le f_i(\mathcal{A}')$, $r(\mathcal{A})\le r(\mathcal{A}')$ for all $j\le i$.
\begin{theorem}\label{comparison-1}
Suppose $\mathcal{A}$ is essential. Given $(\bm\alpha,a),(\bm\alpha',a')\in\mathbb{F}^{d+1}$. If $X_{(\bm\alpha,a)}\subseteq X_{(\bm\alpha',a')}$ for $X_{(\bm\alpha,a)}, X_{(\bm\alpha',a')}\in L(\tilde{\mathcal{A}})$, then
\[
{\rm INV}(\mathcal{A}+H_{\bm\alpha,a})\le{\rm INV}(\mathcal{A}+H_{\bm\alpha',a'}).
\]
\end{theorem}

\autoref{classification-1} says that given $X\in L(\tilde{\mathcal{A}})$, $\chi(\mathcal{A}+H_{\bm\alpha,a},t)$ are the same polynomial for any $(\bm\alpha,a)\in M(\tilde{\mathcal{A}}/X)$, denoted by $\chi(X,t)$. When $\mathbb{F}$ is defined over a finite field $\mathbb{F}_q$ or a rational arrangement, we have the following decomposition formula.
\begin{theorem}\label{Convolution-Formula}
Suppose $\mathcal{A}$ is essential. If (a) $\mathbb{F}=\mathbb{F}_q$ or (b) $\mathbb{F}=\mathbb{R}$ and $\mathcal{A}$ is a rational arrangement, then
\[
\sum_{X\in L(\tilde{\mathcal{A}})}\chi(\tilde{\mathcal{A}}/X,t)\chi(X,t)=t^d(t-1)\chi(\mathcal{A},t).
\]
\end{theorem}
Below is a small example to give a brief illustration of the above results.
\begin{example}\label{example-1}
{\rm Let $\mathcal{A}=\{H_1: x_1=0, H_2: x_1=1, H_3: x_2=0  \}$ be a hyperplane arrangement in $\mathbb{R}^2$. Clearly
\[
L_{0}(\mathcal{A})=\{{\bm v}_1=(0,0),{\bm v}_2=(1,0)\}
\]
and
\[
L_1(\mathcal{A}^o)=\{\mathbb{F}{\bm u}_1,\mathbb{F}{\bm u}_2\} \quad {\rm with} \quad {\bm u}_1=(1,0),{\bm u}_2=(0,1).
\]
From {\rm(\ref{extension-adjoint-arrangement})}, we have that $\tilde{\mathcal{A}}$ is a linear arrangement in $\mathbb{R}^3$ consisting of
\[
H_{\tilde{{\bm u}}_1}: x_1=0,\quad
 H_{\tilde{{\bm u}}_2}: x_2=0 ,\quad H_{\tilde{{\bm v}}_1}:x_3=0,\quad    H_{\tilde{{\bm v}}_2}:x_1-x_3=0.
\]
The intersection lattice $L(\tilde{\mathcal{A}})$ is graded with
\[
L_0(\tilde{\mathcal{A}})=\big\{{\bm 0}\big\},\quad L_1(\tilde{\mathcal{A}})=\big\{H_{\tilde{{\bm u}}_1}\cap H_{\tilde{{\bm u}}_2}, H_{\tilde{{\bm u}}_1}\cap H_{\tilde{{\bm v}}_1}\cap H_{\tilde{{\bm v}}_2}, H_{\tilde{{\bm u}}_2}\cap H_{\tilde{{\bm v}}_1}, H_{\tilde{{\bm u}}_2}\cap H_{\tilde{{\bm v}}_2}\big\},
\]
\[
L_2(\tilde{\mathcal{A}})=\big\{H_{\tilde{{\bm u}}_1}, H_{\tilde{{\bm u}}_2}, H_{\tilde{{\bm v}}_1}, H_{\tilde{{\bm v}}_2}\big\},\quad L_3(\tilde{\mathcal{A}})=\big\{\mathbb{R}^3\big\},
\]
where $L_k(\tilde{\mathcal{A}})$ consists of $k$-dimensional members of $\tilde{\mathcal{A}}$. Suppose  $({\bm\alpha},a)\in M(\tilde{\mathcal{A}}/X)$ for some $X\in L(\tilde{\mathcal{A}})$. When $X$ runs over $L(\tilde{\mathcal{A}})$, there are six classes of $\mathcal{A}+ H_{\bm\alpha,a}$ except that $X\in L_0(\tilde{\mathcal{A}})$, where the red line is the hyperplane $ H_{\bm\alpha,a}$, see the figures below, where Fig.\ref{Fig.1} is for  $X\in L_1(\tilde{\mathcal{A}})$,  Fig.\ref{Fig.2} for  $X=H_{\tilde{{\bm u}}_1}$, Fig.\ref{Fig.3} for  $X=H_{\tilde{{\bm u}}_2}$, Fig.\ref{Fig.4} for  $X=H_{\tilde{{\bm v}}_1}$, Fig.\ref{Fig.5} for  $X=H_{\tilde{{\bm v}}_2}$, and Fig.\ref{Fig.6} for  $X\in L_3(\tilde{\mathcal{A}})$.
\begin{figure}[htbp]
	\centering

		\begin{minipage}[c]{0.15\linewidth}
			\centering
			\includegraphics[width=0.8in]{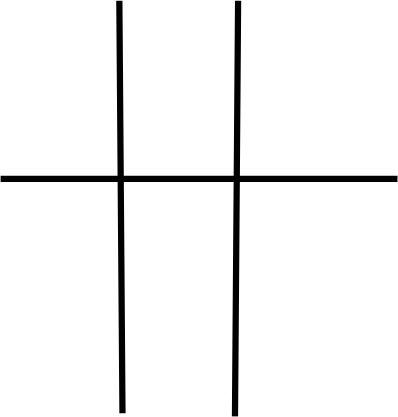}
			\caption{Fig.1} \label{Fig.1}
		\end{minipage}
        \begin{minipage}[c]{0.15\linewidth}
			\centering
			\includegraphics[width=0.8in]{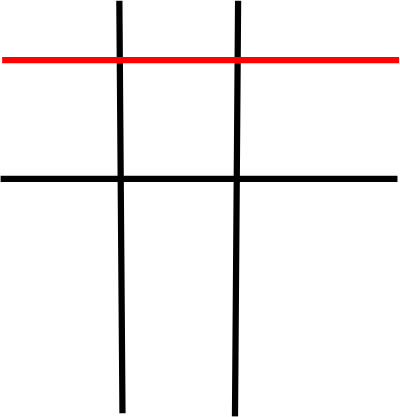}
			\caption{Fig.2} \label{Fig.2}
		\end{minipage}
		\begin{minipage}[c]{0.15\linewidth}
			\centering
			\includegraphics[width=0.8in]{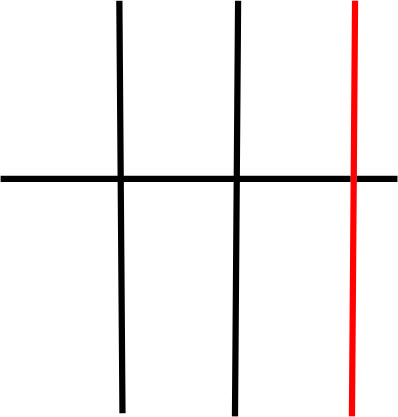}
			\caption{Fig.3} \label{Fig.3}
		\end{minipage}
		\begin{minipage}[c]{0.15\linewidth}
			\centering
			\includegraphics[width=0.8in]{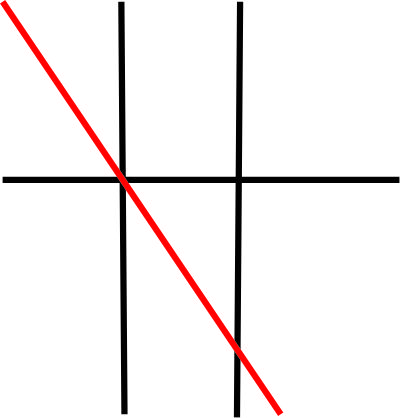}
			\caption{Fig.4} \label{Fig.4}
		\end{minipage}
		\begin{minipage}[c]{0.15\linewidth}
			\centering
			\includegraphics[width=0.8in]{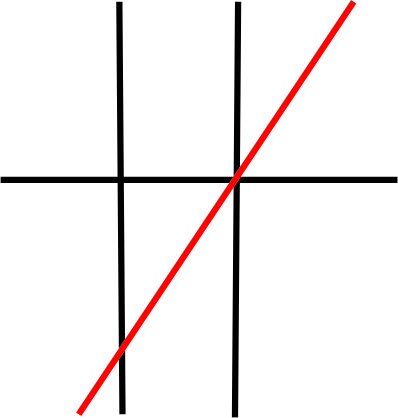}
			\caption{Fig.5} \label{Fig.5}
		\end{minipage}
		\begin{minipage}[c]{0.15\linewidth}
			\centering
			\includegraphics[width=0.8in]{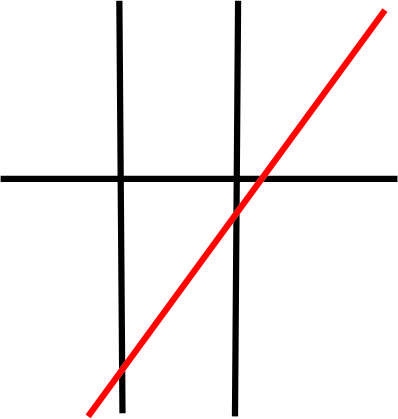}
			\caption{Fig.6} \label{Fig.6}
		\end{minipage}
\end{figure}

Their Whitney polynomials are
\begin{eqnarray*}
 w(\mathcal{A}+H_{\bm\alpha,a})&=&0\;\quad\quad\quad\quad\quad\quad\quad\quad\quad\quad\If  X\in L_0(\tilde{\mathcal{A}}),\\
 w(\mathcal{A}+H_{\bm\alpha,a})&=&3st-4s+t^2-3t+2\quad\;\,\If  X\in L_1(\tilde{\mathcal{A}}),\\
 w(\mathcal{A}+H_{\bm\alpha,a})&=&4st-8s+t^2-4t+4\quad\;\,\If  X=H_{\tilde{{\bm u}}_1},\\
 w(\mathcal{A}+H_{\bm\alpha,a})&=&4st-6s+t^2-4t+3\quad\;\,\If  X=H_{\tilde{{\bm u}}_2},\\
 w(\mathcal{A}+H_{\bm\alpha,a})&=&4st-7s+t^2-4t+4\quad\;\,\If X\in \{H_{\tilde{{\bm v}}_1}, H_{\tilde{{\bm v}}_2}\},\\
 w(\mathcal{A}+H_{\bm\alpha,a})&=&4st-10s+t^2-4t+5\quad\If X=\mathbb{R}^3.
\end{eqnarray*}
We can easily see that $c_{ij}(\mathcal{A}+H_{\bm\alpha,a})\le c_{ij}(\mathcal{A}+H_{\bm\alpha',a'})$ if $X_{(\bm\alpha,a)}\subseteq X_{(\bm\alpha',a')}$ with $X_{(\bm\alpha,a)},X_{(\bm\alpha',a')}\in L(\tilde{\mathcal{A}})$ .
It is easy to check that the remaining combinatorial invariants are order-preserving with respect to $L(\tilde{\mathcal{A}})$.}
\end{example}
\section{Classifications}\label{SEC2}
\paragraph{3.1. Proof of \autoref{classification-1}.}  For any $X\in L(\mathcal{A})$, the rank $\rank_{ \mathcal{A}}(X)$ of $X$ is defined as
\begin{equation*}\label{rank-flat}
\rank_{ \mathcal{A}}(X)=d-{\rm dim}(X).
\end{equation*}
Given a subset $\mathcal{B}\subseteq\mathcal{A}$, define $\rank_{\mathcal{A}}(\mathcal{B})$ to be
\begin{equation}\label{rank}
\rank_{ \mathcal{A}}(\mathcal{B})=\rank \big\{{\bm \alpha}_i\mid H_{\bm\alpha_i,a_i}\in \mathcal{B}\big\}.
\end{equation}
Obviously $\rank_{ \mathcal{A}}(X)=\rank_{ \mathcal{A}}(\mathcal{A}_{  X})$. Generally, if $\mathcal{B}\subseteq\mathcal{A}$ and $X=\bigcap_{H\in \mathcal{B}}H\in L(\mathcal{A})$, then
\begin{equation}\label{rank-identity}
\rank_{ \mathcal{A}}(\mathcal{B})=\rank_{ \mathcal{A}}(X)=d-{\rm dim}(X).
\end{equation}
Below is a key property to verifying \autoref{classification-1}.
\begin{lemma}\label{lemma1}
Suppose $\mathcal{A}$ is essential. Given $({\bm\alpha},a),({\bm\alpha'},a')\in \mathbb{F}^{d+1}$. If $X_{(\bm\alpha,a)}\subseteq X_{(\bm\alpha',a')}$ for $X_{(\bm\alpha,a)},X_{(\bm\alpha',a')}\in L(\tilde{\mathcal{A}})$, then
\[
\rank_{  \mathcal{A}+ H_{\bm\alpha,a}}(\mathcal{A}_{  Y}+ H_{\bm\alpha,a})\le\rank_{ \mathcal{A}+H_{\bm\alpha',a'}}(\mathcal{A}_{  Y}+H_{\bm\alpha',a'}),\quad Y\in L(\mathcal{A}).
\]
\end{lemma}
\begin{proof}
For simplicity, we denote $\mathcal{A}_Z^o:=(\mathcal{A}^o)_Z$ for $Z\in L(\mathcal{A}^o)$. Note from the definition of the linearization that for any $Y\in L(\mathcal{A})$, we have $(\mathcal{A}_Y)^o\subseteq \mathcal{A}^o$ and $Y^o=\bigcap_{H\in (\mathcal{A}_Y)^o}H\in L(\mathcal{A}^o)$. It follows from \eqref{rank} and \eqref{rank-identity} that
\[
\rank_{  \mathcal{A}+ H_{\bm\alpha,a}}\big(\mathcal{A}_{  Y}+ H_{\bm\alpha,a}\big)=\rank_{  \mathcal{A}^o+H_{\bm\alpha}}\big((\mathcal{A}_{  Y})^o+H_{\bm\alpha}\big)=\rank_{  \mathcal{A}^o+H_{\bm\alpha}}\big(\mathcal{A}_{Y^o}^o+H_{\bm\alpha}\big).
\]
Thus, the proof can reduced to verifying that for any $Z\in L(\mathcal{A}^o)$,
\[
\rank_{ \mathcal{A}^o+H_{\bm\alpha}}(\mathcal{A}_Z^o+H_{\bm\alpha})\le\rank_{ \mathcal{A}^o+H_{{\bm \alpha}'}}(\mathcal{A}_Z^o+H_{{\bm \alpha}'}).
\]
Notice from (\ref{complement-partition}) that there exist unique $X_1,X_1'\in L(\tilde{\mathcal{A}_1})$ such that $({\bm\alpha},a)\in M(\tilde{\mathcal{A}_1}/X_1)$ and $({\bm\alpha}',a')\in M(\tilde{\mathcal{A}_1}/X_1')$. Recall from (\ref{A_01}) that $\tilde{\mathcal{A}_1}$ is a subarrangement $\tilde{\mathcal{A}}$ in $\mathbb{F}^{d+1}$, which implies $L(\tilde{\mathcal{A}_1})\subseteq L(\tilde{\mathcal{A}})$. So we have $X_1,X_1'\in L(\tilde{\mathcal{A}})$. Moreover, the minimality of $X_{(\bm\alpha',a')}\in L(\tilde{\mathcal{A}})$ containing $({\bm\alpha}',a')$ means $X_{(\bm\alpha',a')}\subseteq X_1'$.  Since $X_{(\bm\alpha,a)}\subseteq X_{(\bm\alpha',a')}$ and $({\bm\alpha},a)\in X_{(\bm\alpha,a)}$, we obtain $({\bm\alpha},a)\in X_1'$. Likewise, the minimality of $X_1\in L(\tilde{\mathcal{A}_1})$ containing $({\bm\alpha},a)$ implies $X_1\subseteq X_1'$. Consider the hyperplane $H_{\tilde{\bm 0}}:x_{d+1}=0$. Then for any $U\in L(\tilde{\mathcal{A}_1})$, $({\bm\beta},b)\in U$ if and only if $({\bm\beta},0)\in U\cap H_{\tilde{\bm 0}}$ obviously. Hence, we arrive at $({\bm\alpha},0)\in X_1\cap H_{\tilde{\bm 0}}$ and $({\bm\alpha}',0)\in X_1'\cap H_{\tilde{\bm 0}}$. Let
\[
\tilde{\mathcal{A}}_1/H_{\tilde{\bm 0}}:=\{H\cap H_{\tilde{\bm 0}}\ne\emptyset\mid H\in \tilde{\mathcal{A}}_1\}
\]
be the hyperplane arrangement in $H_{\tilde{\bm 0}}$. Then $X_1\cap H_{\tilde{\bm 0}}$ and $X_1'\cap H_{\tilde{\bm 0}}$ are inclusion-minimal members in $L(\tilde{\mathcal{A}}_1/H_{\tilde{\bm 0}})$ containing $({\bm\alpha},0)$ and $({\bm\alpha}',0)$ respectively. Since  $\mathcal{A}$ is essential, we assume $\mathbb{F}^d:=\span\{{\bm\alpha}_1,\ldots,{\bm\alpha}_d\}$ spanned by ${\bm\alpha}_1,\ldots,{\bm\alpha}_d$,  and $Z=\bigcap_{i=1}^{k}H_{\bm\alpha_i}\in L(\mathcal{A}^o)$ if $\rank_{\mathcal{A}^o}(\mathcal{A}^o_Z)=k$. It is easily seen that for each $j\in [d]$, $\bigcap_{i\in [d];i\ne j}H_{\bm \alpha_i}\in L_1(\mathcal{A}^o)$. Assume $\mathbb{F}\bm u_j=\bigcap_{i\in [d];i\ne j}H_{\bm \alpha_i}$. Then the vectors $\bm u_j$ for $j\in [d]$ are linearly independent and the normal space $Z^\perp$ of $Z$ can be represented as
\[
Z^{\perp}=\span\{{\bm\alpha}_{1},\ldots,{\bm\alpha}_{k}\}=\bigcap_{j=k+1}^dH_{\bm u_j}\in L(\sigma\mathcal{A}^o),
\]
since ${\bm \alpha}_i\cdot\bm u_j=0$ for all $i=1,\ldots,k$ and $j=k+1\ldots,d$. Let
\[
\rank_{ \mathcal{A}^o+H_{\bm\alpha}}\big(\mathcal{A}_Z^o+H_{\bm\alpha}\big)=r\quad
\And\quad \rank_{ \mathcal{A}^o+H_{{\bm \alpha}'}}\big(\mathcal{A}_Z^o+
H_{{\bm \alpha}'}\big)=r'.
\]
We need to prove $r\le r'$. Since $\rank_{ \mathcal{A}^o}(\mathcal{A}^o_Z)=k$, we have $k\le r, r' \le \min\{k+1,d\}$. The case $k=d$ is trivial. If $k<d$, we will give a proof by contradiction. Suppose $r> r'$. Immediately, we have $r=k+1$ and $r'=k$. Note the fact that for any ${\bm \beta}\in \mathbb{F}^d$, $\rank_{ \mathcal{A}^o+H_{\bm\beta}}\big(\mathcal{A}_Z^o+H_{\bm\beta}\big)= \rank_{ \mathcal{A}^o}(\mathcal{A}^o_Z)$ if and only if $\bm\beta\in Z^{\perp}$. Then we have ${\bm \alpha}'\in Z^{\perp}$ and ${\bm \alpha}\notin Z^{\perp}$. Since the isomorphism $L(\sigma\mathcal{A}^o)\cong L(\tilde{\mathcal{A}}_1/H_{\tilde{\bm 0}})$
sends $Z^{\perp}$ to $Z^{\perp}\times \{0\}$, we obtain $({\bm \alpha}',0)\in Z^{\perp}\times \{0\}$. From the minimality of $X_1'\cap H_{\tilde{\bm 0}}\in L(\tilde{\mathcal{A}}_1/H_{\tilde{\bm 0}})$ containing $({\bm \alpha}',0)$, we have $X_1'\cap H_{\tilde{\bm 0}}\subseteq Z^{\perp}\times \{0\}$. Thus $({\bm \alpha},0)\in X_1\cap H_{\tilde{\bm 0}}\subseteq X_1'\cap H_{\tilde{\bm 0}}\subseteq Z^{\perp}\times \{0\}$, contradicting ${\bm \alpha}\notin Z^{\perp}$. It completes the proof.
\end{proof}
The following property will be used to prove \autoref{classification-1}.
\begin{lemma}\label{lemma0}
Suppose $\mathcal{A}$ is essential. Given $({\bm\alpha},a),({\bm\alpha'},a')\in \mathbb{F}^{d+1}$. If $X_{(\bm\alpha,a)}\subseteq X_{(\bm\alpha',a')}$ for $X_{(\bm\alpha,a)},X_{(\bm\alpha',a')}\in L(\tilde{\mathcal{A}})$, then
\[
H_{\bm\alpha',a'}\cap L_0(\mathcal{A})\subseteq H_{\bm\alpha,a}\cap L_0(\mathcal{A}).
\]
\end{lemma}
\begin{proof}
If $H_{\bm\alpha',a'}\cap L_0(\mathcal{A})=\emptyset$, the case is trivial. So, it is enough to show that for any ${\bm v}_k\in H_{\bm\alpha',a'}\cap L_0(\mathcal{A})$, ${\bm v}_k$ belongs to $H_{\bm\alpha,a}$. Note from ${\bm v}_k\in H_{\bm\alpha',a'}\cap L_0(\mathcal{A})$ that ${\bm\alpha}'\cdot{\bm v}_k=a'$. Together with the definition of $\tilde{\mathcal{A}}$ in \eqref{extension-adjoint-arrangement}, we have $({\bm\alpha}',a')\in H_{\tilde{\bm v}_k}$. In addition, the minimality of $X_{(\bm\alpha',a')}$ containing $({\bm\alpha}',a')$ yields $X_{(\bm\alpha',a')}\subseteq H_{\tilde{\bm v}_k}$. Applying the assumption $X_{(\bm\alpha,a)}\subseteq X_{(\bm\alpha',a')}$, we obtain $X_{(\bm\alpha,a)}\subseteq H_{\tilde{\bm v}_k}$. It follows from $({\bm\alpha},a)\in X_{(\bm\alpha,a)}$ that $({\bm\alpha},a)\in H_{\tilde{\bm v}_k}$. Hence, we have ${\bm\alpha}\cdot{\bm v}_k=a$, i.e., ${\bm v}_k\in H_{\bm\alpha,a}$. We finish this proof.
\end{proof}
\begin{proof}[Proof of \autoref{classification-1}] Given $({\bm\alpha},a),({\bm\alpha}',a')\in M(\tilde{\mathcal{A}}/X)$ for some $X\in L(\tilde{\mathcal{A}})$. Note that $L(\mathcal{A}+H_{\bm\alpha,a})$ and $L(\mathcal{A}+H_{\bm\alpha',a'})$ can be written as the following forms
\[
L(\mathcal{A}+ H_{\bm\alpha,a})=L(\mathcal{A})\sqcup L_{\bm\alpha,a},\quad L(\mathcal{A}+ H_{\bm\alpha',a'})=L(\mathcal{A})\sqcup L_{\bm\alpha',a'},
\]
where $L_{\bm\alpha,a}=\big\{Y\cap H_{\bm\alpha,a}\ne\emptyset\mid Y\in L(\mathcal{A}), Y\cap H_{\bm\alpha,a}\notin L(\mathcal{A})\big\}$ and
$L_{\bm\alpha',a'}=\big\{Y\cap H_{\bm\alpha',a'}\ne\emptyset\mid Y\in L(\mathcal{A}), Y\cap H_{\bm\alpha',a'}\notin L(\mathcal{A})\big\}$.
Define a map $\iota:\, L(\mathcal{A}+H_{\bm\alpha,a})\to L(\mathcal{A}+H_{\bm\alpha',a'})$ such that
\[
Y\mapsto Y \mbox{ if }Y\in L(\mathcal{A}), \And Y\cap H_{\bm\alpha,a}\mapsto Y\cap H_{\bm\alpha',a'}\mbox{ if } Y\cap H_{\bm\alpha,a}\in L_{\bm\alpha,a}.
\]
The proof of  \autoref{classification-1} reduces to showing that $\iota$ is an order-preserving bijection. Next we will check that $\iota$ is well-defined, order-preserving and bijective in turn.

To verify that $\iota$ is well-defined, from the definition of the map $\iota$, we only need to show that $Y\cap H_{\bm\alpha',a'}\in L_{\bm\alpha',a'}$ whenever $Y\cap H_{\bm\alpha,a}\in L_{\bm\alpha,a}$. Taking $Z=Y\cap H_{\bm\alpha,a}\in L_{\bm\alpha,a}$ with $Y\in L(\mathcal{A})$, we have $\dim(Z)=\dim(Y)-1$.
It follows from \eqref{rank-identity} and  \autoref{lemma1} that
\[\rank_{ \mathcal{A}+H_{\bm\alpha',a'}}\big(\mathcal{A}_{  Y}+H_{\bm\alpha',a'}\big)=\rank_{ \mathcal{A}+ H_{\bm\alpha,a}}\big(\mathcal{A}_{  Y}+ H_{\bm\alpha,a}\big)=\rank_{ \mathcal{A}+ H_{\bm\alpha,a}}\big(\mathcal{A}_{  Y}\big)+1.
\]
So we have $Y\cap H_{\bm\alpha',a'}\ne\emptyset$, $\iota(Z)\in L(\mathcal{A}+H_{\bm\alpha',a'})$ and $\dim(\iota(Z))=\dim(Y)-1$ from the elementary linear algebra. Next we will prove $\iota(Z)\notin L(\mathcal{A})$. Suppose $\iota(Z)\in  L(\mathcal{A})$, then there exists $H_{\bm\alpha_j,a_j}\in\mathcal{A}\setminus\mathcal{A}_{  Y}$ satisfying $\iota(Z)=Y\cap H_{\bm\alpha_j,a_j}$. Then we have $\iota(Z)\cap L_0(\mathcal{A})=Z\cap L_0(\mathcal{A})$ from \autoref{lemma0}. This implies that $\iota(Z)\cap L_0(\mathcal{A})\subseteq H_{\bm\alpha,a}$. Note that $\iota(Z)\cap L_0(\mathcal{A})\ne\emptyset$ since $\mathcal{A}$ is essential. So $\iota(Z)\cap H_{\bm\alpha,a}\ne\emptyset$. We claim $\iota(Z)\subseteq H_{\bm\alpha,a}$. Otherwise, we have $\dim(\iota(Z)\cap H_{\bm\alpha,a})=\dim(\iota(Z))-1$. Similarly,
from \eqref{rank-identity} and  \autoref{lemma1}, we have
\[\rank_{ \mathcal{A}+H_{\bm\alpha',a'}}\big(\mathcal{A}_{  \iota(Z)}+H_{\bm\alpha',a'}\big)=\rank_{ \mathcal{A}+ H_{\bm\alpha,a}}\big(\mathcal{A}_{  \iota(Z)}+ H_{\bm\alpha,a}\big)=\rank_{ \mathcal{A}+ H_{\bm\alpha,a}}\big(\mathcal{A}_{  \iota(Z)}\big)+1.
\]
This means that $\dim(\iota(Z)\cap H_{\bm\alpha',a'})=\dim(\iota(Z))-1$, which is in contradiction with $\iota(Z)=\iota(Z)\cap H_{\bm\alpha',a'}$. Note from $\iota(Z)\subseteq Y$ that $\iota(Z)\subseteq Y\cap H_{\bm\alpha,a}=Z$. It follows from $\dim(Z)=\dim(\iota(Z))$ that $Z=\iota(Z)\in L(\mathcal{A})$, contradicting $Z\notin L(\mathcal{A})$. So we obtain $Y\cap H_{\bm\alpha',a'}\in L_{\bm\alpha',a'}$, i.e., $\iota$ is well-defined. Clearly $\iota$ is an order-preserving map from its definition.

By symmetry of $H_{\bm\alpha,a}$ and $H_{\bm\alpha',a'}$, to prove the bijectivity of $\iota$, it is sufficient to show that $\iota$ is injective. Equivalently, from the definition of $\iota$, it is enough to show that the map $\iota$ from $L_{\bm\alpha,a}$ to $L_{\bm\alpha',a'}$ is an injection. Using the same argument as $Z$ and $\iota(Z)$, given $Z_1=Y_1\cap H_{\bm\alpha,a}\in L_{\bm\alpha,a}$ with $Y_1\in L(\mathcal{A})$, we have $\dim(\iota(Z_1))=\dim(Y_1)-1$. Therefore,  if $\dim(Y_1)\ne\dim(Y)$, then $\iota(Z_1)\ne\iota(Z)$. The injectivity of $\iota$ can be reduced to verifying that if $Z_1\ne Z$  and $\dim(Y_1)=\dim(Y)$, then $\iota(Z_1)\ne\iota(Z)$. Suppose $\iota(Z_1)=\iota(Z)$, i.e., $Y_1\cap H_{\bm\alpha',a'}=Y\cap H_{\bm\alpha',a'}$. This means that
\[
Y_1\cap H_{\bm\alpha',a'}=Y_1\cap H_{\bm\alpha',a'}\cap H\quad \For\quad H\in\mathcal{A}_{  Y}\setminus\mathcal{A}_{  Y_1}.
\]
Note from $Z\ne Z_1$ that $Y_1\ne Y$. So there is $H_{\bm\alpha_k,a_k}\in\mathcal{A}_{  Y}\setminus\mathcal{A}_{  Y_1}$ such that $Y_1\cap H_{\bm\alpha',a'}=Y_1\cap H_{\bm\alpha_k,a_k}\in L(\mathcal{A})$, which contradicts $Y_1\cap H_{\bm\alpha',a'}\notin L(\mathcal{A})$. Hence, $\iota$ is bijective. We complete the proof.
\end{proof}
Now we are turning to the theme of how to classify one-element extensions of a hyperplane arrangement which may be not essential. Recall that $O$ is the space spanned by the normal vectors ${\bm\alpha}_1,{\bm\alpha}_2,\ldots,{\bm\alpha}_m$ of the hyperplanes in $\mathcal{A}$. In fact, to study the one-element extensions $\mathcal{A}+H_{\bm\alpha,a}$, we only need to consider the case that  $\mathcal{A}$ is essential. Indeed, if $\mathcal{A}$ is not essential, then $O$ is a proper subspace of $\mathbb{F}^d$. Given ${\bm \alpha}\in \mathbb{F}^d\setminus O$. From the elementary linear algebra, clearly $X\cap H_{\bm\alpha,a}\ne \emptyset,X$ for any $X\in L(\mathcal{A})$ and $a\in\mathbb{F}$. It implies that
\[
L\big(\mathcal{A}+ H_{\bm\alpha,a}\big)\cong L(\mathcal{A})\times C_2,
\]
where $C_2$ is a 2-chain and $(X, x)\le (X',x')$ in $L(\mathcal{A})\times C_2$ if and only if $X\le X'$ in $L(\mathcal{A})$ and $x\le x'$ in $C_2$. Namely, for all ${\bm \alpha}\in\mathbb{F}^d\setminus O$, the intersection semi-lattices $L(\mathcal{A}+ H_{\bm\alpha,a})$ are mutually isomorphic. So it remains to consider $\mathcal{A}+ H_{\bm\alpha,a}$ in the case ${\bm \alpha}\in O$, which is equivalent to studying one-element extensions of the hyperplane arrangement $\mathcal{A}/O:=\{H\cap O\ne\emptyset\mid H\in \mathcal{A}\}$ that is an essential arrangement in $O$.

\paragraph{3.2. More classifications.}
As an application of \autoref{classification-1}, we will investigate more classifications of  several special types of one-element extensions. When $\mathcal{A}=\big\{H_{\bm\alpha_i}\mid i\in[m]\big\}$ is an essential linear arrangement in $\mathbb{F}^d$, obviously $\mathcal{A}^o=\mathcal{A}$ and $L_{0}(\mathcal{A})=\{{\bf 0}\}$. Recall from the definitions of $\sigma\mathcal{A}$ in \eqref{adjoint-arrangement} and  $\tilde{\mathcal{A}}$ in \eqref{extension-adjoint-arrangement} that
\[
\sigma\mathcal{A}=\big\{H_{\bm u_i}:\bm u_i\cdot{\bm x}=0\mid i\in[n_{1}]\big\}
\]
and
\[
\tilde{\mathcal{A}}=\{H_{\tilde{\bf 0}}:x_{d+1}=0\}\cup\tilde{\mathcal{A}_1},\;{\rm where\;}\tilde{\mathcal{A}}_1=\big\{H_{\tilde{\bm u}_i}:\tilde{\bm u}_i\cdot{\bm x}=0\mid i\in[n_{1}]\big\}.
\]
Clearly $L(\sigma\mathcal{A})\cong L(\tilde{\mathcal{A}}/H_{\tilde{\bm 0}})\cong L(\tilde{\mathcal{A}}_1)$ in this case. Most recently, Fu-Wang \cite[Theorem 3.2]{Fu-Wang2021} have obtained the classification of the linear one-element extensions $\mathcal{A}+H_{\bm\alpha}$ of $\mathcal{A}$ for all $\bm\alpha\in\mathbb{F}^d$ by the adjoint arrangement $\sigma\mathcal{A}$. In fact, it can be regarded as a special case of \autoref{classification-1}. Next we will present a short proof of \autoref{classification-2}.
\begin{corollary}[\cite{Fu-Wang2021}]\label{classification-2}
Let $\mathcal{A}=\big\{H_{\bm\alpha_i}\mid i\in[m]\big\}$ be an essential linear arrangement in $\mathbb{F}^{d}$. Given $X\in L(\sigma\mathcal{A})$. For any vectors ${\bm \alpha},{\bm \alpha'}\in  M(\sigma\mathcal{A}/X)$, we have
\[
L(\mathcal{A}+ H_{\bm \alpha})\cong L(\mathcal{A}+H_{\bm\alpha'}).
\]
\end{corollary}
\begin{proof}
Note that the isomorphism $L(\sigma\mathcal{A})\cong L\big(\tilde{\mathcal{A}}/H_{\tilde{\bf0}}\big)$ sends $X\in L(\sigma\mathcal{A})$ to $X\times\{0\}\in L\big(\tilde{\mathcal{A}}/H_{\tilde{\bf0}}\big)$. It follows from \eqref{complement-partition} that
\[
\mathbb{F}^d=\bigsqcup_{X\in L(\sigma\mathcal{A})}M(\sigma\mathcal{A}/X)\;\And\;H_{\tilde{\bm0}}=\bigsqcup_{Y\in L(\tilde{\mathcal{A}}/H_{\tilde{\bm0}})} M\big((\tilde{\mathcal{A}}/H_{\tilde{\bm0}})/Y\big)=\bigsqcup_{X\in L(\sigma\mathcal{A})} M\big(\tilde{\mathcal{A}}/X\times\{0\}\big).
\]
It means that $\bm\alpha\in M(\sigma\mathcal{A}/X)$ if and only if $(\bm\alpha,0)\in M\big(\tilde{\mathcal{A}}/X\times\{0\}\big)$. Since $L(\tilde{\mathcal{A}}/H_{\tilde{\bm0}})\subseteq L(\tilde{\mathcal{A}})$ and $H_{\bm\alpha}=H_{\bm\alpha,0}$,  \autoref{classification-1} implies \autoref{classification-2}.
\end{proof}

Suppose $\mathcal{A}$ is an essential linear arrangement in $\mathbb{F}^d$. As another byproduct of \autoref{classification-1}, the partition $\mathbb{F}^{d+1}=\bigsqcup_{X\in L(\sigma\mathcal{A})}\big(M(\sigma\mathcal{A}/X)\times\{0\}\bigsqcup M(\sigma\mathcal{A}/X)\times\{\mathbb{F}\setminus\{0\}\}\big)$ classifies the one-element extensions of the linear arrangement $\mathcal{A}$ as follows.
\begin{corollary}\label{classification-3}
Let $\mathcal{A}=\big\{H_{\bm\alpha_i}\mid i\in[m]\big\}$ be an essential linear arrangement in $\mathbb{F}^{d}$. Given $X\in L(\sigma\mathcal{A})$. For any vectors ${\bm \alpha},{\bm \alpha'}\in  M(\sigma\mathcal{A}/X)$ and $a,a'\in\mathbb{F}\setminus\{0\}$, we have
\[
L(\mathcal{A}+ H_{\bm \alpha})\cong L(\mathcal{A}+H_{\bm\alpha'})\quad\And\quad
L(\mathcal{A}+ H_{\bm\alpha,a})\cong L(\mathcal{A}+H_{\bm\alpha',a'}).
\]
\end{corollary}
\begin{proof}
Clearly the first isomorphism in \autoref{classification-3} follows from \autoref{classification-2} directly. Next we will prove the second isomorphism in \autoref{classification-3}. Note from $\bigcap_{H\in L(\tilde{\mathcal{A}}_1)}H=0^d\times\mathbb{F}$ that the isomorphism $L(\sigma\mathcal{A})\cong L(\tilde{\mathcal{A}}_1)$ sends $X\in L(\sigma\mathcal{A})$ to $X\times\mathbb{F}\in L(\tilde{\mathcal{A}}_1)$. In this case, we have $L(\tilde{\mathcal{A}})=L(\tilde{\mathcal{A}}_1)\sqcup L(\tilde{\mathcal{A}}/H_{\tilde{\bm0}})$ obviously. Together with $H_{\tilde{\bm0}}=\bigsqcup_{X\in L(\tilde{\mathcal{A}}/H_{\tilde{\bm0}})}M(\tilde{\mathcal{A}}/X)$, it follows from decomposition \eqref{complement-partition} that
\[
\mathbb{F}^{d+1}\setminus H_{\tilde{\bm0}}=\bigsqcup_{Y\in L(\tilde{\mathcal{A}}_1)}M(\tilde{\mathcal{A}}/Y)=\bigsqcup_{X\in L(\sigma\mathcal{A})}M(\tilde{\mathcal{A}}/X\times\mathbb{F}) \;\And\; \mathbb{F}^d=\bigsqcup_{X\in L(\sigma\mathcal{A})}M(\sigma\mathcal{A}/X).
\]
Immediately, the above equalities imply that for any $(\bm\alpha,a)\in\mathbb{F}^d\times\{\mathbb{F}\setminus\{0\}\}=\mathbb{F}^{d+1}\setminus H_{\tilde{\bm0}}$, $\bm\alpha\in M(\sigma\mathcal{A}/X)$ if and only if $(\bm\alpha,a)\in M(\tilde{\mathcal{A}}/X\times\mathbb{F})$. So \autoref{classification-1} means that the second isomorphism in \autoref{classification-3} holds as well.
\end{proof}

In end of this section, we will further give the classification of the linear one-element extensions of the hyperplane arrangement $\mathcal{A}$. Recall
$L_{0}(\mathcal{A})=\{{\bm v}_1,{\bm v}_2,\ldots,{\bm v}_{n_0}\}$. Let
\[
\bar{\mathcal{A}}:=\mathcal{A}_0\cup\sigma\mathcal{A}^o
\]
be a linear arrangement in $\mathbb{F}^d$, where $\mathcal{A}_0=\big\{H_{\bm v_i}:\bm v_i\cdot\bm x=0\mid i\in[n_0]\big\}$. Likewise, the whole space $\mathbb{F}^d$ has a partition $\mathbb{F}^d=\bigsqcup_{X\in L(\bar{\mathcal{A}})}M(\bar{\mathcal{A}}/X)$, which determines the following classification.
\begin{corollary}\label{classification-4}
Suppose $\mathcal{A}$ is essential. Given $X\in L(\bar{\mathcal{A}})$. For any ${\bm\alpha},{\bm\alpha'}\in  M(\bar{\mathcal{A}}/X)$, we have
\[
L(\mathcal{A}+ H_{\bm \alpha})\cong L(\mathcal{A}+H_{\bm\alpha'}).
\]
\end{corollary}
\begin{proof}
Consider the hyperplane $H_{\tilde{\bm0}}:x_{d+1}=0$. Let $\tilde{\mathcal{A}}/H_{\tilde{\bm0}}:=\{H\cap H_{\tilde{\bm0}}\ne\emptyset\mid H\in\tilde{\mathcal{A}}\setminus H_{\tilde{\bm0}}\}$ ($\tilde{\mathcal{A}}$ may not contain $H_{\tilde{\bm0}}$.). Recall from the definition $\tilde{\mathcal{A}}$ in \eqref{extension-adjoint-arrangement} that clearly $L(\bar{\mathcal{A}})$ is isomorphic to $L(\tilde{\mathcal{A}}/H_{\tilde{\bm0}})$ with sending $X\in L(\bar{\mathcal{A}})$ to $X\times\{0\}\in L(\tilde{\mathcal{A}}/H_{\tilde{\bm0}})$. It implies that $\bm\alpha\in M\big(\bar{\mathcal{A}}/X\big)$ if and only if $(\bm\alpha,0)\in M\big((\tilde{\mathcal{A}}/H_{\tilde{\bm0}})/X\times\{0\}\big)$. So, to prove the corollary, we only need to show that
for any $(\bm\alpha,0),(\bm\alpha',0)\in M\big((\tilde{\mathcal{A}}/H_{\tilde{\bm0}})/X\times\{0\}\big)$ the result holds. Recall from \eqref{complement-partition} that given $\bm x\in H_{\tilde{\bm0}}$, there are unique elements $Y\in L(\tilde{\mathcal{A}})$ and $Z\in L(\tilde{\mathcal{A}}/H_{\tilde{\bm0}})$ such that $\bm x\in M(\tilde{\mathcal{A}}/Y)$ and $\bm x\in M\big((\tilde{\mathcal{A}}/H_{\tilde{\bm 0}})/Z\big)$ respectively. In this case, we claim that
\[
M\big((\tilde{\mathcal{A}}/H_{\tilde{\bm 0}})/Z\big)\subseteq M(\tilde{\mathcal{A}}/Y).
\]
This says that for any $(\bm\alpha,0),(\bm\alpha',0)\in M\big((\tilde{\mathcal{A}}/H_{\tilde{\bm 0}})/Z\big)$, we have $(\bm\alpha,0),(\bm\alpha',0)\in M(\tilde{\mathcal{A}}/Y)$. Immediately \autoref{classification-1} indicates $
L(\mathcal{A}+H_{\bm\alpha})\cong L(\mathcal{A}+H_{\bm\alpha'})$. Hence, to complete the proof, it remains to verify the claim. Note from $\bm x\in H_{\tilde{\bm 0}}$ and $\bm x\in M(\tilde{\mathcal{A}}/Y)$ that $Y\cap H_{\tilde{\bm 0}}\in L(\tilde{\mathcal{A}}/H_{\tilde{\bm 0}})$ contains $\bm x$ and $\bm x\notin Y\cap H$ for all $H\in\tilde{\mathcal{A}}\setminus\tilde{\mathcal{A}}_{  Y}$. It implies that $\bm x\notin \bigcup_{H\in\tilde{\mathcal{A}}\setminus\tilde{\mathcal{A}}_{  Y}} Y\cap H_{\tilde{\bm0}}\cap H$ and
\begin{equation}\label{partition-1}
M\big((\tilde{\mathcal{A}}/H_{\tilde{\bm0}})/Y\cap H_{\tilde{\bm0}}\big)=\big(Y\cap H_{\tilde{\bm 0}}\big)\setminus\bigcup_{H\in\tilde{\mathcal{A}}\setminus\tilde{\mathcal{A}}_Y}(Y\cap H_{\tilde{\bm0}}\cap H)=\big(Y\cap H_{\tilde{\bm 0}}\big)\setminus\bigcup_{H\in\tilde{\mathcal{A}}\setminus\tilde{\mathcal{A}}_Y}H.
\end{equation}
Hence, we have $\bm x\in M\big((\tilde{\mathcal{A}}/H_{\tilde{\bm0}})/Y\cap H_{\tilde{\bm0}}\big)$. Immediately the uniqueness of $Z$ implies $Z=Y\cap H_{\tilde{\bm 0}}$. Suppose $\bm y\in M\big((\tilde{\mathcal{A}}/H_{\tilde{\bm 0}})/Z\big)$, then $\bm y\in \big(Y\cap H_{\tilde{\bm 0}}\big)\setminus\bigcup_{H\in\tilde{\mathcal{A}}\setminus\tilde{\mathcal{A}}_{  Y}}H$ from \eqref{partition-1}. It follows from $Y\cap H_{\tilde{\bm 0}}\subseteq Y$ that $\bm y\in Y\setminus\bigcup_{H\in\tilde{\mathcal{A}}\setminus\tilde{\mathcal{A}}_{  Y}}H$, i.e., $\bm y\in M(\tilde{\mathcal{A}}/Y)$. Hence, we obtain $M\big((\tilde{\mathcal{A}}/H_{\tilde{\bm 0}})/Z\big)\subseteq M(\tilde{\mathcal{A}}/Y)$,
which finishes the proof.
\end{proof}

\section{Uniform comparisons}\label{SEC3}
\paragraph{4.1. Proof of \autoref{comparison-1}.}
To prove \autoref{comparison-1}, it requires Whitney's celebrated NBC (no broken circuit) theorem \cite{Whitney1932}, which is an important tool to computing the unsigned Whitney number of the first kind. For any subset $J\subseteq[m]$ with $\bigcap_{j\in J}H_{\bm\alpha_j,a_j}\ne\emptyset$, if $d-{\rm dim}(\bigcap_{j\in J}H_{\bm\alpha_j,a_j})=\#J$, then it is called {\em affinely independent}, otherwise called {\em affinely dependent} with respect to $\mathcal{A}$. It is easily seen that subsets $J$ with $\bigcap_{j\in J}H_{\bm\alpha_j,a_j}=\emptyset$ are irrelevant to affine independence and dependence. Moreover, given a subset $J\subseteq[m]$, define $\rank_{ \mathcal{A}}(J)$ (with respect to $\mathcal{A}$) to be
\[
\rank_{ \mathcal{A}}(J):=\max\big\{\#I\mid I\subseteq J\And I\mbox{ is affinely independent with respect to }\mathcal{A}\big\}.
\]
In particular, if $\bigcap_{j\in J}H_{\bm\alpha_j,a_j}\ne\emptyset$, the maximal affinely independent subset $I$ of $J$ is called a {\em basis} of $J$, which means  $\rank_{ \mathcal{A}}(J)=\#I$ and $\bigcap_{j\in I}H_{\bm\alpha_j,a_j}=\bigcap_{j\in J}H_{\bm\alpha_j,a_j}$. Clearly, $\rank_{ \mathcal{A}}(J)\le \#J$. If $\bigcap_{j\in J}H_{\bm\alpha_j,a_j}\ne\emptyset$, then $J$ is affinely independent (affinely dependent resp.) if and only if $\rank_{\mathcal{A}}(J)=\#J$ ($\rank_{ \mathcal{A}}(J)<\#J$ resp.) with respect to $\mathcal{A}$. Additionally, a minimal affinely dependent subset $I$ of $[m]$ is said to be an {\em affine circuit} with respect to $\mathcal{A}$, i.e., $\bigcap_{i\in I}H_{\bm\alpha_i,a_i}\ne\emptyset$ and $\rank_{ \mathcal{A}}(I)={\rm rank}_{ \mathcal{A}}\big(I\setminus \{i\}\big)=\#I-1$ for all $i\in I$. Given a total order $\prec$ on $[m]$, an {\em affine broken circuit} is a subset of $[m]$ obtained from an affine circuit by deleting  the minimal element. An {\em affine NBC set} is a subset of $[m]$ containing no affine broken circuits. It follows that every affine NBC set is affinely independent. Denote by ${\rm NBC}_{k}(\mathcal{A})$ the collection of affine NBC $k$-subsets of $[m]$. Here the affine NBC sets are exactly the $\chi$-independent sets defined in \cite[p. 72]{Orlik-Terao1992}, which gives a counting interpretation of the unsigned Whitney number $w_k^+(\mathcal{A})$ of the first kind. Generally, a hyperplane arrangement can be regarded as a semimatroid. For more information about affine broken circuits of hyperplane arrangements can refer to \cite{Forge-Zaslavsky2016}.
\begin{theorem}[Affine NBC Theorem \cite{Orlik-Terao1992}]\label{affine-NBC} Let $\mathcal{A}$ be a hyperplane arrangement in a $d$-dimensional vector space $V$. Then
\[
 w_k^+(\mathcal{A})=\#{\rm NBC}_k(\mathcal{A}),\quad k=0,1,\ldots,d.
\]
\end{theorem}
Here we may consider that given $X\in L(\mathcal{A})$, $\mathcal{A}/X$ is a multi-arrangement in $X$ consisting of hyperplane $H\cap X\ne\emptyset$ for all $H\in\mathcal{A}\setminus\mathcal{A}_{  X}$. In addition, the label of each hyperplane in $\mathcal{A}/X$ is inherited from $\mathcal{A}$, and these labels of hyperplanes with respect to $\mathcal{A}/X$ inherit the total order on $[m]$ with respect to $\mathcal{A}$. Applying \autoref{affine-NBC} to equation \eqref{Whitney-Polynomial-I}, together with \eqref{Whitney-Polynomial}, immediately we obtain
\begin{equation}\label{NBC-Theorem}
c_{ij}(\mathcal{A})=\sum_{X\in L(\mathcal{A}),\,\dim(X)=i}\#{\rm NBC}_j(\mathcal{A}/X).
\end{equation}
To verify \autoref{comparison-1}, we need the following lemmas.
\begin{lemma}\label{NBC-lemma1}
Suppose $\mathcal{A}$ is essential. Given $(\bm\alpha_{m+1},a_{m+1}),(\bm\alpha'_{m+1},a'_{m+1})\in\mathbb{F}^{d+1}$. If $X_{(\bm\alpha_{m+1},a_{m+1})}\subseteq X_{(\bm\alpha'_{m+1},a'_{m+1})}$ for $X_{(\bm\alpha_{m+1},a_{m+1})}, X_{(\bm\alpha'_{m+1},a'_{m+1})}\in L(\tilde{\mathcal{A}})$, then
\[
\rank_{ \mathcal{A}+H_{\bm\alpha_{m+1},a_{m+1}}}\big(J\sqcup\{m+1\}\big)\le \rank_{ \mathcal{A}+H_{\bm\alpha'_{m+1},a'_{m+1}}}\big(J\sqcup\{m+1\}\big), \; J\subseteq[m].
\]
\end{lemma}
\begin{proof}
Given $J\subseteq[m]$. Clearly
\[
\rank_{ \mathcal{A}+H_{\bm\alpha_{m+1},a_{m+1}}}(J\sqcup\{m+1\}),\rank_{ \mathcal{A}+H_{\bm\alpha'_{m+1},a'_{m+1}}}(J\sqcup\{m+1\})
\in\{\rank_{ \mathcal{A}}(J),\rank_{ \mathcal{A}}(J)+1\}.
\]
Suppose $\rank_{ \mathcal{A}+H_{\bm\alpha_{m+1},a_{m+1}}}(J\sqcup\{m+1\})>\rank_{ \mathcal{A}+H_{\bm\alpha'_{m+1},a'_{m+1}}}(J\sqcup\{m+1\})$. In other words, $\rank_{ \mathcal{A}+H_{\bm\alpha_{m+1},a_{m+1}}}(J\sqcup\{m+1\})=\rank_{ \mathcal{A}}(J)+1$ and $\rank_{ \mathcal{A}+H_{\bm\alpha'_{m+1},a'_{m+1}}}(J\sqcup\{m+1\})=\rank_{ \mathcal{A}}(J)$. This implies that there is a maximal affinely independent subset $I$ of $J$ with respect to $\mathcal{A}$ such that $\bigcap_{i\in I}H_{\bm\alpha_i,a_i}\bigcap H_{\bm\alpha_{m+1},a_{m+1}}\ne\emptyset$,  $\bigcap_{i\in I}H_{\bm\alpha_i,a_i}$ and $\bigcap_{i\in I}H_{\bm\alpha_i,a_i}\bigcap H_{\bm\alpha'_{m+1},a'_{m+1}}=\emptyset$ or $\bigcap_{i\in I}H_{\bm\alpha_i,a_i}\subseteq H_{\bm\alpha'_{m+1},a'_{m+1}}$. Let $Y=\bigcap_{i\in I}H_{\bm\alpha_i,a_i}$. Recall from \eqref{rank} that we can easily obtain
\[
\rank_{ \mathcal{A}+H_{\bm\alpha_{m+1},a_{m+1}}}(\mathcal{A}_{  Y}+H_{\bm\alpha_{m+1},a_{m+1}})=\#I+1>\#I=\rank_{ \mathcal{A}+H_{\bm\alpha_{m+1},a_{m+1}}}(\mathcal{A}_{  Y}+H_{\bm\alpha'_{m+1},a'_{m+1}}).
\]
This is impossible from \autoref{lemma1}, which completes the proof.
\end{proof}
Given $(\bm\alpha_{m+1},a_{m+1}),(\bm\alpha_{m+1}',a_{m+1}')\in \mathbb{F}^{d+1}$. For convenience, for $i\in[m]$, we denote by $H_{\bm\alpha'_i,a'_i}=H_{\bm\alpha_i,a_i}$ in $\mathcal{A}+H_{\bm\alpha'_{m+1},a'_{m+1}}$ in this subsection. Given an element $Y\in L(\mathcal{A}+H_{\bm\alpha_{m+1},a_{m+1}})$. Let
\[
J_{  Y}=\big\{i\in[m+1]\mid H_{\bm\alpha_i,a_i}\in\mathcal{A}+H_{\bm\alpha_{m+1},a_{m+1}},\, Y\subseteq H_{\bm\alpha_i,a_i}\big\}.
\]
Obviously $\bigcap_{i\in J_Y}H_{\bm\alpha_i,a_i}=Y$. Given a basis $I_Y$ of $J_Y$ with respect to $\mathcal{A}+H_{\bm\alpha_{m+1},a_{m+1}}$. Let $Y'=\bigcap_{i\in I_Y}H_{\bm\alpha'_i,a'_i}$. Notice from \autoref{NBC-lemma1} that we have $\rank_{ \mathcal{A}+H_{\bm\alpha'_{m+1},a'_{m+1}}}(I_Y)=\# I_Y$, i.e., $I_Y$ is affinely independent with respect to $\mathcal{A}+H_{\bm\alpha'_{m+1},a'_{m+1}}$. So $Y'\in L(\mathcal{A}+H_{\bm\alpha'_{m+1},a'_{m+1}})$, i.e., the construction is well-defined. Likewise, let $J_{Y'}=\big\{i\in[m+1]\mid H_{\bm\alpha'_i,a'_i}\in\mathcal{A}+H_{\bm\alpha'_{m+1},a'_{m+1}},\,Y'\subseteq H_{\bm\alpha'_i,a'_i}\big\}$. Then $\rank_{ \mathcal{A}+H_{\bm\alpha'_{m+1},a'_{m+1}}}(J_{Y'})=\rank_{ \mathcal{A}+H_{\bm\alpha_{m+1},a_{m+1}}}(J_Y)=\#I_Y$. With this construction, we will further present the following lemmas.
\begin{lemma}\label{NBC-lemma2}
Suppose $\mathcal{A}$ is essential. Given $(\bm\alpha_{m+1},a_{m+1}),(\bm\alpha'_{m+1},a'_{m+1})\in\mathbb{F}^{d+1}$.
\begin{itemize}
  \item [\i] Given $Y\in L(\mathcal{A}+H_{\bm\alpha_{m+1},a_{m+1}})$ and $Y'=\bigcap_{i\in I_Y }H_{\bm\alpha'_i,a'_i}$ for some basis $I_Y$ of $J_{  Y}$. If $X_{(\bm\alpha_{m+1},a_{m+1})}\subseteq X_{(\bm\alpha'_{m+1},a'_{m+1})}$ for $X_{(\bm\alpha_{m+1},a_{m+1})}, X_{(\bm\alpha'_{m+1},a'_{m+1})}\in L(\tilde{\mathcal{A}})$, then $J_{Y'}\subseteq J_Y$.
  \item [\ii] Let $Y_1,Y_2\in L(\mathcal{A}+H_{\bm\alpha_{m+1},a_{m+1}})$ with the same rank, and $Y_i'=\bigcap_{j\in I_{  Y_i}}H_{\bm\alpha'_j,a'_j}$ for some basis $I_{Y_i}$ of $J_{Y_i}$ with $i=1,2$, then $Y_1'\ne Y_2'$ whenever $Y_1\ne Y_2$.
\end{itemize}

 \end{lemma}
\begin{proof}
Suppose $k\in J_{  Y'}\setminus J_{  Y}$, i.e., $Y'\subseteq H_{\bm\alpha'_k,a'_k}$ and $Y\nsubseteq H_{\bm\alpha_k,a_k}$. If $m+1\notin I_{  Y}$, then $Y=Y'$ and $k=m+1$ by $Y\nsubseteq H_{\bm\alpha_k,a_k}$. It follows from \autoref{lemma0} that $Y\cap H_{\bm\alpha_{m+1},a_{m+1}}\cap L_0(\mathcal{A})=Y'\cap L_0(\mathcal{A})\ne\emptyset$ since $\mathcal{A}$ is essential. This implies that $I_{  Y}\sqcup\{m+1\}$ is affinely independent with respect to $\mathcal{A}+H_{\bm\alpha_{m+1},a_{m+1}}$. So we have
\begin{equation}\label{rank-inequality}
\rank_{ \mathcal{A}+H_{\bm\alpha_{m+1},a_{m+1}}}(I_{  Y}\sqcup\{m+1\})=\#I_{  Y}+1>\#I_{  Y}=\rank_{ \mathcal{A}+H_{\bm\alpha'_{m+1},a'_{m+1}}}(I_{  Y}\sqcup\{m+1\}).
\end{equation}
From \autoref{NBC-lemma1}, this is impossible. If $m+1\in I_{  Y}$, let $Z=\bigcap_{i\in I_Y\setminus\{m+1\}}H_{\bm\alpha_i,a_i}$. Then $k\ne m+1$. We deduce that $Y=Z\cap H_{\bm\alpha_{m+1},a_{m+1}}$ and $Y'=Z\cap H_{\bm\alpha'_{m+1},a'_{m+1}}$ and $Y,Y'\subsetneqq Z$ in this case. Note from $Y'\subseteq H_{\bm\alpha_k,a_k},\;Y\nsubseteq H_{\bm\alpha_k,a_k}$ that $Z\nsubseteq H_{\bm\alpha_k,a_k}$ and $Z\cap H_{\bm\alpha'_{m+1},a'_{m+1}}=Z\cap H_{\bm\alpha_k,a_k}=Y'\in L(\mathcal{A})$. Immediately,
we obtain from \autoref{lemma0} that $Y\cap L_0(\mathcal{A})=Y'\cap L_0(\mathcal{A})\ne\emptyset$ since $\mathcal{A}$ is essential.
Analogous to the inequality \eqref{rank-inequality}, this yields
\[
\rank_{ \mathcal{A}+H_{\bm\alpha_{m+1},a_{m+1}}}(I_{  Y}\sqcup\{k\})=\#I_{  Y}+1>\#I_{  Y}=\rank_{ \mathcal{A}+H_{\bm\alpha'_{m+1},a'_{m+1}}}(I_{  Y}\sqcup\{k\}).
\]
Likewise, it is impossible by \autoref{NBC-lemma1}. Hence, we complete the proof of the first part.

For the second part, suppose $Y_1'=Y_2'$, then we have $J_{  Y_1'}=J_{  Y_2'}$. From the first part, we arrive at $I_{  Y_2}\subseteq J_{  Y_2'}\subseteq J_{  Y_1}$. This can give rise to $Y_2=\bigcap_{i\in I_{Y_2}}H_{\bm\alpha_i,a_i}=Y_1$ since $Y_1,Y_2$ have the same rank, contradicting $Y_1\ne Y_2$.
\end{proof}
\begin{lemma}\label{NBC-lemma3}
Suppose $\mathcal{A}$ is essential. Given $(\bm\alpha_{m+1},a_{m+1}),(\bm\alpha'_{m+1},a'_{m+1})\in\mathbb{F}^{d+1}$, $Y\in L(\mathcal{A}+H_{\bm\alpha_{m+1},a_{m+1}})$ and $Y'=\bigcap_{i\in I_Y}H_{\bm\alpha'_i,a'_i}$ for some basis $I_Y$ of $J_Y$.  If $X_{(\bm\alpha_{m+1},a_{m+1})}\subseteq X_{(\bm\alpha'_{m+1},a'_{m+1})}$ for $X_{(\bm\alpha_{m+1},a_{m+1})}, X_{(\bm\alpha'_{m+1},a'_{m+1})}\in L(\tilde{\mathcal{A}})$ and $J\subseteq[m+1]\setminus J_Y$, then
\[
\rank_{ \mathcal{A}+H_{\bm\alpha_{m+1},a_{m+1}}/Y}(J)\le \rank_{ \mathcal{A}+H_{\bm\alpha'_{m+1},a'_{m+1}}/Y'}(J).
\]
\end{lemma}
\begin{proof}
Note that for any $J\subseteq[m+1]\setminus J_{  Y}$, $\rank_{ \mathcal{A}+H_{\bm\alpha_{m+1},a_{m+1}}/Y}(J)$ equals the maximal cardinality of the affinely independent subsets $I$ of $J$ with respect to $\mathcal{A}+H_{\bm\alpha_{m+1},a_{m+1}}/Y$,
where $I$ is affinely independent with respect to $\mathcal{A}+H_{\bm\alpha_{m+1},a_{m+1}}/Y$ if and only if $\bigcap_{i\in I}H_{\bm\alpha_i,a_i}\bigcap Y\ne\emptyset$ and $\dim(Y)-\dim\big(\bigcap_{i\in I}H_{\bm\alpha_i,a_i}\bigcap Y\big)=\#I$.

Given $J\subseteq [m+1]\setminus J_{  Y}$, we have $J\subseteq[m+1]\setminus J_{  Y}\subseteq[m+1]\setminus J_{  Y'}$ from \autoref{NBC-lemma2}. Let $I$ be a maximal affinely independent subset of $J$ with respect to $\mathcal{A}+H_{\bm\alpha_{m+1},a_{m+1}}/Y$ satisfying $\rank_{ \mathcal{A}+H_{\bm\alpha_{m+1},a_{m+1}}/Y}(J)=\#I$. To obtain the result, it is enough to show that $\# I= \rank_{ \mathcal{A}+H_{\bm\alpha'_{m+1},a'_{m+1}}/Y'}(I)$ since $\rank_{ \mathcal{A}+H_{\bm\alpha'_{m+1},a'_{m+1}}/Y'}(I)\le\rank_{ \mathcal{A}+H_{\bm\alpha'_{m+1},a'_{m+1}}/Y'}(J)$. Notice that
\begin{eqnarray}\label{NBC-rank}
\#I&=&\rank_{ \mathcal{A}+H_{\bm\alpha_{m+1},a_{m+1}}/Y}(J)={\rm dim}(Y)-{\rm dim}\big(\bigcap_{i\in I\cup I_{  Y}}H_{\bm\alpha_i,a_i}\big)\nonumber\\
&=&d-\rank_{ \mathcal{A}+H_{\bm\alpha_{m+1},a_{m+1}}}(J_{  Y})-\big(d-\rank_{ \mathcal{A}+H_{\bm\alpha_{m+1},a_{m+1}}}(I\cup I_{  Y})\big)\nonumber\\
&=&\rank_{ \mathcal{A}+H_{\bm\alpha_{m+1},a_{m+1}}}(I\cup I_{  Y})-\rank_{ \mathcal{A}+H_{\bm\alpha_{m+1},a_{m+1}}}(J_{  Y}).
\end{eqnarray}
It follows from $\rank_{ \mathcal{A}+H_{\bm\alpha_{m+1},a_{m+1}}}(J_{  Y})=\#I_{  Y}$ that $\rank_{ \mathcal{A}+H_{\bm\alpha_{m+1},a_{m+1}}}(I\cup I_{  Y})=\#I+\#I_{  Y}$. Immediately, we have $\rank_{ \mathcal{A}+H_{\bm\alpha_{m+1},a_{m+1}}}(I\cup I_{  Y})\le\rank_{ \mathcal{A}+H_{\bm\alpha'_{m+1},a'_{m+1}}}(I\cup I_{  Y})\le\#I+\#I_{  Y}$ from $\autoref{NBC-lemma1}$. This yields $\rank_{ \mathcal{A}+H_{\bm\alpha'_{m+1},a'_{m+1}}}(I\cup I_{  Y})=\#I+\#I_{  Y}$, i.e., $I\cup I_{  Y}$ is affinely independent with respect to $\mathcal{A}+H_{\bm\alpha'_{m+1},a'_{m+1}}$. In analogy with the argument of equation \eqref{NBC-rank}, we arrive at
\[
\rank_{ \mathcal{A}+H_{\bm\alpha'_{m+1},a'_{m+1}}/Y'}(I)=\rank_{ \mathcal{A}+H_{\bm\alpha'_{m+1},a'_{m+1}}}(I\cup I_{  Y})-\rank_{ \mathcal{A}+H_{\bm\alpha'_{m+1},a'_{m+1}}}(J_{  Y'}).
\]
Since $\rank_{ \mathcal{A}+H_{\bm\alpha'_{m+1},a'_{m+1}}}(J_{  Y'})=\# I_{  Y}$, we get $\rank_{ \mathcal{A}+H_{\bm\alpha'_{m+1},a'_{m+1}}/Y'}(I)=\#I$.
\end{proof}
The following result is a key property to verifying \autoref{comparison-1}.
\begin{lemma}\label{NBC-lemma4}
Suppose $\mathcal{A}$ is essential. Given $(\bm\alpha_{m+1},a_{m+1}),(\bm\alpha'_{m+1},a'_{m+1})\in\mathbb{F}^{d+1}$, $Y\in L(\mathcal{A}+H_{\bm\alpha_{m+1},a_{m+1}})$ and $Y'=\bigcap_{i\in I_Y}H_{\bm\alpha'_i,a'_i}$ for some basis $I_Y$ of $J_Y$.  If $X_{(\bm\alpha_{m+1},a_{m+1})}\subseteq X_{(\bm\alpha'_{m+1},a'_{m+1})}$ for $X_{(\bm\alpha_{m+1},a_{m+1})}, X_{(\bm\alpha'_{m+1},a'_{m+1})}\in L(\tilde{\mathcal{A}})$, then
\[
\NBC_j\big(\mathcal{A}+H_{\bm\alpha_{m+1},a_{m+1}}/Y\big)\subseteq\NBC_j\big(\mathcal{A}+H_{\bm\alpha'_{m+1},a'_{m+1}}/Y'\big).
\]
\end{lemma}
\begin{proof}
Given a total order $\prec$ on $[m+1]$ such that for any $k\in[m+1]\setminus(J_{  Y}\setminus J_{  Y'})$ and $k'\in J_{  Y}\setminus J_{  Y'}$, $k\prec k'$. Arguing by contradiction, suppose $J\in \NBC_j\big(\mathcal{A}+H_{\bm\alpha_{m+1},a_{m+1}}/Y\big)\setminus \NBC_j\big(\mathcal{A}+H_{\bm\alpha'_{m+1},a'_{m+1}}/Y'\big)$. Then $J$ is affinely independent with respect to $\mathcal{A}+H_{\bm\alpha_{m+1},a_{m+1}}/Y$ and $\rank_{ \mathcal{A}+H_{\bm\alpha_{m+1},a_{m+1}}/Y}(J)=\#J$. It follows from \autoref{NBC-lemma3} that $\rank_{ \mathcal{A}+H_{\bm\alpha_{m+1},a_{m+1}}/Y}(J)\le\rank_{ \mathcal{A}+H_{\bm\alpha'_{m+1},a'_{m+1}}/Y'}(J)\le\# J$. Immediately, we have $\rank_{ \mathcal{A}+H_{\bm\alpha'_{m+1},a'_{m+1}}/Y'}(J)=\# J$, i.e., $J$ is also affinely independent with respect to $\mathcal{A}+H_{\bm\alpha'_{m+1},a'_{m+1}}/Y'$. Notice from $J\notin\NBC_j\big(\mathcal{A}+H_{\bm\alpha'_{m+1},a'_{m+1}}/Y'\big)$ that there exists $k\in[m+1]\setminus J_{Y}$ such that $J\cup\{k\}$ contains an affine circuit $I$ (with respect to $\mathcal{A}+H_{\bm\alpha'_{m+1},a'_{m+1}}/Y'$) satisfying $k\in I$ and $k\prec k'$ for all $k'\in I\setminus\{k\}$. This means that $k\in[m]$, $\rank_{ \mathcal{A}+H_{\bm\alpha'_{m+1},a'_{m+1}}/Y'}(I)=\#I-1$ and $\bigcap_{i\in I\cup I_{  Y}}H_{\bm\alpha'_i,a'_i}\ne\emptyset$. We assert $\bigcap_{i\in I\cup I_{  Y}}H_{\bm\alpha_i,a_i}\ne\emptyset$. Suppose $\bigcap_{i\in I\cup I_{  Y}}H_{\bm\alpha_i,a_i}=\emptyset$, then we have $m+1\in I\cup I_{  Y}$. Moreover, $m+1\in I_{  Y}$. Otherwise, $m+1\in I\setminus\{k\}$. This implies that $\bigcap_{i\in I\cup I_{  Y}\setminus\{k\}}H_{\bm\alpha'_i,a'_i}=\bigcap_{i\in I\cup I_{  Y}\setminus\{m+1\}}H_{\bm\alpha'_i,a'_i}\in L(\mathcal{A})$. It follows from \autoref{lemma0} that $\bigcap_{i\in I\cup I_Y}H_{\bm\alpha_i,a_i}\bigcap L_0(\mathcal{A})=\bigcap_{i\in I\cup I_{  Y}\setminus\{m+1\}}H_{\bm\alpha'_i,a'_i}\bigcap L_0(\mathcal{A})\ne\emptyset$ since $\mathcal{A}$ is essential, contradicting $\bigcap_{i\in I\cup I_{  Y}}H_{\bm\alpha_i,a_i}=\emptyset$.
So $m+1\in I_{  Y}$. Note from the fact $H_{\bm\alpha_i',a_i'}=H_{\bm\alpha_i,a_i}$ for $i\in[m]$ that we have $\bigcap_{i\in I\cup I_Y\setminus\{m+1\}}H_{\bm\alpha_i,a_i}=\bigcap_{i\in I\cup I_Y\setminus\{k,m+1\}}H_{\bm\alpha_i,a_i}$ since $I$ is an affine circuit. It gives rise to $\bigcap_{i\in I\cup I_{  Y}}H_{\bm\alpha_i,a_i}\ne\emptyset$, which is in contradiction with $\bigcap_{i\in I\cup I_{  Y}}H_{\bm\alpha_i,a_i}=\emptyset$. So we have shown the assertion. This assertion says $\bigcap_{i\in I}H_{\bm\alpha_i,a_i}\bigcap Y\ne\emptyset$. Together with \autoref{NBC-lemma3}, we have $\rank_{ \mathcal{A}+H_{\bm\alpha_{m+1},a_{m+1}}/Y}(I)\le\rank_{ \mathcal{A}+H_{\bm\alpha'_{m+1},a'_{m+1}}/Y}(I)=\#I-1$, which means that $I$ is affinely dependent with respect to $\mathcal{A}+H_{\bm\alpha_{m+1},a_{m+1}}/Y$. Hence, there is an affine broken circuit $I'\setminus\{k\}\subseteq I\setminus\{k\}\subseteq J$ with respect to $\mathcal{A}+H_{\bm\alpha_{m+1},a_{m+1}}/Y$ satisfying $k\in I'$, contradicting $J\in\NBC_j\big(\mathcal{A}+H_{\bm\alpha_{m+1},a_{m+1}}/Y\big)$. It completes the proof.
\end{proof}
With these above results we will give a proof of \autoref{comparison-1}.
\begin{proof}[Proof of \autoref{comparison-1}]
Let $(\bm\alpha_{m+1},a_{m+1})=(\bm\alpha,a)$ and $(\bm\alpha'_{m+1},a'_{m+1})=(\bm\alpha',a')$. \autoref{NBC-lemma2} implies that there is a rank-preserving injection
\[
\phi:L(\mathcal{A}+H_{\bm\alpha,a})\to L(\mathcal{A}+H_{\bm\alpha',a'}),\quad Y\mapsto Y',
\]
where $Y'=\bigcap_{j\in I_Y}H_{\bm\alpha_i',a_i'}\in L(\mathcal{A}+H_{\bm\alpha',a'})$ for some basis $I_{  Y}$ of $Y$. Recall from the definition of the Whitney number of the second kind that the injectivity of $\phi$ implies directly $ W_i(\mathcal{A}+H_{\bm\alpha,a})\le W_i(\mathcal{A}+H_{\bm\alpha',a'})$. Moreover, together with \autoref{NBC-lemma4} and equation \eqref{NBC-Theorem}, we arrive at $c_{  ij}(\mathcal{A}+H_{\bm\alpha,a})\le c_{  ij}(\mathcal{A}+H_{\bm\alpha',a'})$. From $\chi(\mathcal{A}+H_{\bm\alpha,a},t)= w(\mathcal{A}+H_{\bm\alpha,a};0,t)$, we obtain $ w_i^+(\mathcal{A}+H_{\bm\alpha,a})=c_{  di}(\mathcal{A}+H_{\bm\alpha,a})$. Likewise, we have $ w_i^+(\mathcal{A}+H_{\bm\alpha',a'})=c_{  di}(\mathcal{A}+H_{\bm\alpha',a'})$. So we get $ w_i^+(\mathcal{A}+H_{\bm\alpha,a})\le w_i^+(\mathcal{A}+H_{\bm\alpha',a'})$ since $c_{  di}(\mathcal{A}+H_{\bm\alpha,a})\le c_{  di}(\mathcal{A}+H_{\bm\alpha',a'})$. Additionally, combining with the identity on  $f_i(\mathcal{A}+H_{\bm\alpha,a})$ in the end of Subsection 2.1 and equation \eqref{Whitney-Polynomial-I}, we obtain
\[
\sum_{i=0}^{d}f_i(\mathcal{A}+H_{\bm\alpha,a})s^{d-i}=(-1)^d w(\mathcal{A}+H_{\bm\alpha,a};-s,-1).
\]
Notice from \eqref{Whitney-Polynomial} that $ w(\mathcal{A}+H_{\bm\alpha,a};-s,-1)=\Sum_{0\le j\le i\le d}(-1)^dc_{  ij}(\mathcal{A}+H_{\bm\alpha,a})s^{d-i}$. This yields $f_i(\mathcal{A}+H_{\bm\alpha,a})=\sum_{j=0}^{i}c_{  ij}(\mathcal{A}+H_{\bm\alpha,a})$. Likewise, we have $f_i(\mathcal{A}+H_{\bm\alpha',a'})=\sum_{j=0}^{i}c_{  ij}(\mathcal{A}+H_{\bm\alpha',a'})$. Hence, we get $f_i(\mathcal{A}+H_{\bm\alpha,a})\le f_i(\mathcal{A}+H_{\bm\alpha',a'})$ via $c_{  ij}(\mathcal{A}+H_{\bm\alpha,a})\le c_{  ij}(\mathcal{A}+H_{\bm\alpha',a'})$. Immediately, we have $r(\mathcal{A}+H_{\bm\alpha,a})=f_d(\mathcal{A}+H_{\bm\alpha,a})\le f_d(\mathcal{A}+H_{\bm\alpha',a'})=r(\mathcal{A}+H_{\bm\alpha',a'})$ as well.
\end{proof}
\paragraph{4.2. More uniform comparisons.} In this subsection, we will explore more order-preserving relations of combinatorial invariants associated with several special types of one-element extensions. Fu-Wang \cite[Theorem 3.3]{Fu-Wang2021} have shown the following result, which may be viewed as a special case of \autoref{comparison-1}.
\begin{corollary}[\cite{Fu-Wang2021}]\label{comparison-2}
Let $\mathcal{A}=\big\{H_{\bm\alpha_i}\mid i\in[m]\big\}$ be an essential linear arrangement in $\mathbb{F}^{d}$. Given vectors $\bm\alpha,\bm\alpha'\in\mathbb{F}^d$. If $X_{\bm\alpha}\subseteq X_{\bm\alpha'}$ for $X_{\bm\alpha}, X_{\bm\alpha'}\in L(\sigma\mathcal{A})$, then
\[
{\rm INV}(\mathcal{A}+ H_{\bm \alpha})\le{\rm INV}(\mathcal{A}+H_{\bm\alpha'}).
\]
\end{corollary}
\begin{proof}
Note that the isomorphism $L(\sigma\mathcal{A})\cong L\big(\tilde{\mathcal{A}}/H_{\tilde{\bf0}}\big)$ sends $X\in L(\sigma\mathcal{A})$ to $X\times\{0\}\in L\big(\tilde{\mathcal{A}}/H_{\tilde{\bf0}}\big)\subseteq L(\tilde{\mathcal{A}})$. In analogy with the argument of \autoref{classification-2}, we can obtain that $X_{\bm\alpha}\subseteq X_{\bm\alpha'}$ if and only if $X_{\bm\alpha}\times\{0\}\subseteq X_{\bm\alpha'}\times\{0\}$ with $X_{\bm\alpha}\times\{0\}, X_{\bm\alpha'}\times\{0\}\in L(\tilde{\mathcal{A}})$. Moreover, from the proof of \autoref{classification-2}, we have shown that for any $Y\in L(\sigma\mathcal{A})$, $\bm\beta\in M(\sigma\mathcal{A}/Y)$ if and only if $(\bm\beta,0)\in M\big(\tilde{\mathcal{A}}/Y\times\{0\}\big)$. Hence, \autoref{comparison-1} directly leads to the result.
\end{proof}
Analogous to the idea of the proof of \autoref{comparison-2}, together with the argument of \autoref{classification-3} (\autoref{classification-4}, resp.), \autoref{comparison-1} will also imply the following two results (resp.) whose proofs we leave as simple exercises for readers.
\begin{corollary}\label{comparison-3}
Let $\mathcal{A}=\big\{H_{\bm\alpha_i}\mid i\in[m]\big\}$ be an essential linear arrangement in $\mathbb{F}^{d}$. Given vectors $\bm\alpha,\bm\alpha'\in\mathbb{F}^d$ and $a,a'\in\mathbb{F}\setminus\{0\}$. If $X_{\bm\alpha}\subseteq X_{\bm\alpha'}$ for $X_{\bm\alpha}, X_{\bm\alpha'}\in L(\sigma\mathcal{A})$, then
\[
{\rm INV}(\mathcal{A}+ H_{\bm \alpha})\le{\rm INV}(\mathcal{A}+H_{\bm\alpha'}),{\rm INV}(\mathcal{A}+H_{\bm\alpha,a})\,\And\,{\rm INV}(\mathcal{A}+ H_{\bm\alpha,a})\le{\rm INV}(\mathcal{A}+H_{\bm\alpha',a'}).
\]
\end{corollary}
\begin{corollary}\label{comparison-4}
Suppose $\mathcal{A}$ is essential. Given vectors $\bm\alpha,\bm\alpha'\in\mathbb{F}^d$. If $X_{\bm\alpha}\subseteq X_{\bm\alpha'}$ for $X_{\bm\alpha}, X_{\bm\alpha'}\in L(\bar{\mathcal{A}})$, then
\[
{\rm INV}(\mathcal{A}+ H_{\bm \alpha})\le{\rm INV}(\mathcal{A}+H_{\bm\alpha'}).
\]
\end{corollary}
\section{Proof of \autoref{Convolution-Formula}}\label{SEC4-0}
To obtain \autoref{Convolution-Formula}, we need to introduce the method based on finite field for evaluating the characteristic polynomial of hyperplane arrangement defined over some finite field $\mathbb{F}_q$. A {\em rational arrangement} $\mathcal{A}$ in $\mathbb{R}^d$ consists of hyperplanes $H$ having the following form
\begin{equation*}\label{rational-arrangement}
H:\bm\alpha\cdot\bm x=a\quad\For\quad (\bm\alpha,a)\in\mathbb{Z}^d\times\mathbb{Z}.
\end{equation*}
Then the subset $H_q$ in $\mathbb{F}_q^d$ consists of all $d$-tuples $(x_1,\ldots,x_d)$ satisfying the defining equation of $H$ in $\mathbb{F}_q$, where $q=p^r$ for some prime number $p$. For a large enough prime number $p$, it automatically yields a hyperplane arrangement $\mathcal{A}_q$ in $\mathbb{F}_q^d$ consisting of the hyperplanes $H_q$ associated with $\mathcal{A}$ such that $L(\mathcal{A})\cong L(\mathcal{A}_q)$, see \cite[Proposition 5.13]{Stanley2007}. The following result was first implicit in the work of Crapo and Rota \cite{Crapo-Rota1970} in 1970.  At the end of the 20th Century,  Athanasiadis \cite{Athanasiadis1996} and Bj\"{o}rner and Ekedahl \cite{Bjorner-Ekedahl1997} showed that the characteristic polynomial $\chi(\mathcal{A},q)$ exactly counts the number of all points in $\mathbb{F}_q^d$ that do not lie in any of hyperplanes in $\mathcal{A}$.
\begin{theorem}{\rm\cite{Athanasiadis1996,Bjorner-Ekedahl1997}}\label{characteristic-field}
Let $\mathcal{A}$ be a hyperplane arrangement in $\mathbb{F}^d$.
\begin{itemize}
\item[\i]  If $\mathbb{F}$ is a finite field $\mathbb{F}_q$, then $\chi(\mathcal{A},q)=\#\big(\mathbb{F}_q^d\setminus\bigcup_{H\in\mathcal{A}}H\big)$.
\item[\ii] If $\mathcal{A}$ is a rational arrangement in $\mathbb{R}^d$ and $L(\mathcal{A})\cong L(\mathcal{A}_q)$ for some prime power $q$, then
$\chi(\mathcal{A},q)=\chi(\mathcal{A}_q,q).$
\end{itemize}
\end{theorem}
Notice from the definition of the induced adjoint arrangement in \eqref{extension-adjoint-arrangement} that when $\mathcal{A}$ is rational, the induced adjoint arrangement $\tilde{\mathcal{A}}$ is a rational arrangement as well. Below is an important lemma to verifying \autoref{Convolution-Formula}, which follows from \cite[Proposition 5.13]{Stanley2007} directly.
\begin{lemma}\label{cong1}
Let $\mathcal{A}$ be an essential rational arrangement in $\mathbb{R}^d$. For a large enough prime number $p$ and a fixed integer vector $(\bm\alpha,a)\in\mathbb{Z}^{d+1}$,
\[
L(\mathcal{A})\cong L(\mathcal{A}_p),\; L(\tilde{\mathcal{A}})\cong L(\tilde{\mathcal{A}_p}) \And L(\mathcal{A}+H_{\bm\alpha,a})\cong L\big((\mathcal{A}+H_{\bm\alpha,a})_p\big)
\]
holds.
\end{lemma}

\begin{proof}[Proof of \autoref{Convolution-Formula}]
For the case (a), let $\Omega$ be the set consisting of pairs $\big((\bm\alpha,a),\bm x\big)\in\mathbb{F}_q^{d+1}\times\mathbb{F}_q^d$ such that for fixed $(\bm \alpha,a)\in\mathbb{F}_q^{d+1}$, the points $\bm x$ do not lie in any of the hyperplanes in $\mathcal{A}+H_{\bm\alpha,a}$. Namely,
\[
\Omega=\big\{\big((\bm\alpha,a),\bm x\big)\mid (\bm\alpha,a)\in\mathbb{F}_q^{d+1}\And\bm x\in M(\mathcal{A}+H_{\bm\alpha,a})\big\}.
\]
Next we will count the number of members of $\Omega$ in two different ways. When the origin does not belong to $M(\mathcal{A})$, we have
\begin{eqnarray}\label{EQ}
\#\Omega&=&\#\Big\{\big((\bm\alpha,a),\bm x\big)\mid (\bm\alpha,a)\in\mathbb{F}_q^{d+1}=\bigsqcup_{X\in L(\tilde{\mathcal{A}})}M(\tilde{\mathcal{A}}/X) \And \bm x\in M(\mathcal{A}+H_{\bm\alpha,a})\Big\}\notag\\
&=&\sum_{X\in L(\tilde{\mathcal{A}})}\#\big\{\big((\bm\alpha,a),\bm x\big)\mid (\bm\alpha,a)\in M(\tilde{\mathcal{A}}/X) \And \bm x\in M(\mathcal{A}+H_{\bm\alpha,a})\big\}\notag\\
&=&\sum_{X\in L(\tilde{\mathcal{A}})}\sum_{(\bm\alpha,a)\in M(\tilde{\mathcal{A}}/X)}\#M(\mathcal{A}+H_{\bm\alpha,a}).
\end{eqnarray}
Recall from \autoref{classification-1} that $\chi(X,q)=\chi(\mathcal{A}+H_{\bm\alpha,a},q)$ for any $(\bm\alpha,a)\in M(\tilde{\mathcal{A}}/X)$. Applying \i in \autoref{characteristic-field} to \eqref{EQ}, we obtain
\begin{eqnarray}\label{left1}
\#\Omega=\sum_{X\in L(\tilde{\mathcal{A}})}\sum_{(\bm\alpha,a)\in M(\tilde{\mathcal{A}}/X)}\chi(\mathcal{A}+H_{\bm\alpha,a},q)=\sum_{X\in L(\tilde{\mathcal{A}})}\chi(\tilde{\mathcal{A}}/X,q)\chi(X,q).
\end{eqnarray}
On the other hand, noting that for any $(\bm\alpha,a)\in\mathbb{F}_q^{d+1}$, $\bm x\in M\big(\mathcal{A}+H_{\bm\alpha,a}\big)$ if and only if $\bm x\in M(\mathcal{A})$ and $\bm\alpha\cdot\bm x\ne a$,  then $\Omega$ can be written as
\begin{eqnarray*}
\Omega=\big\{\big((\bm\alpha,a),\bm x\big)\mid (\bm\alpha,a)\in\mathbb{F}_q^{d+1},\;\bm x\in M(\mathcal{A})\And \bm\alpha\cdot\bm x\ne a\big\}.
\end{eqnarray*}
It is easily seen that
\begin{eqnarray*}
\#\Omega&=&\sum_{\bm x\in M(\mathcal{A})}\#\big\{\big((\bm\alpha,a),\bm x\big)\mid (\bm\alpha,a)\in\mathbb{F}_q^{d+1}\And \bm\alpha\cdot\bm x\ne a\big\}\\
&=&\sum_{\bm x\in M(\mathcal{A})}\sum_{a=0}^{q-1}\#\big\{\big((\bm\alpha,a),\bm x\big)\mid \bm\alpha\in\mathbb{F}_q^d\And \bm\alpha\cdot\bm x\ne a\big\}\\
&=&\sum_{\bm x\in M(\mathcal{A})}q(q^d-q^{d-1})=q^d(q-1)\# M(\mathcal{A}).
\end{eqnarray*}
It follows from \i in \autoref{characteristic-field} that $\#\Omega=q^d(q-1)\chi(\mathcal{A},q)$. Together with \eqref{left1}, we get
\begin{equation}\label{finite-convolutionformula}
\sum_{X\in L(\tilde{\mathcal{A})}} \chi(\tilde{\mathcal{A}}/X,q)\chi(X,q)=q^d(q-1)\chi(\mathcal{A},q).
\end{equation}
When the origin belongs to $M(\mathcal{A})$, we can obtain \eqref{finite-convolutionformula} by using the similar argument as the first case. More generally, for any finite field $\mathbb{F}_{q'}$ containing $\mathbb{F}_q$, we can treat  $\mathcal{A}$ as a hyperplane arrangement over $\mathbb{F}_{q'}$. Then the equality \eqref{finite-convolutionformula} holds for  infinitely many $q'$.  The fundamental theorem of algebra says that if two polynomials agree on infinitely many values, then they are the same polynomial. This means that it is a polynomial identity in $t$.

For the case (b), when $\mathcal{A}$ is a rational arrangement in $\mathbb{R}^d$, the induced adjoint arrangement $\tilde{\mathcal{A}}$ is a rational arrangement in $\mathbb{R}^{d+1}$ as well. For any $X\in L(\tilde{\mathcal{A}})$, we can choose a representative $(\bm\alpha_X,a_X)\in\mathbb{Q}^{d+1}$ of $M(\tilde{\mathcal{A}}/X)$ based on the denseness of rational numbers. Noting that $(\bm\alpha_X,a_X)\in M(\tilde{\mathcal{A}}/X)$ if and only if $(n\bm\alpha_X,na_X)\in M(\tilde{\mathcal{A}}/X)$ for any positive integers $n$, then we always assume $(\bm\alpha_X,a_X)\in\mathbb{Z}^{d+1}$ from now on. Applying \autoref{cong1}, for a large enough prime number $p$, we have
\begin{equation}\label{cong2}
L(\mathcal{A})\cong L(\mathcal{A}_p),\; L(\tilde{\mathcal{A}})\cong L(\tilde{\mathcal{A}_p}) \And L(\mathcal{A}+H_{\bm\alpha_X,a_X})\cong L\big((\mathcal{A}+H_{\bm\alpha_X,a_X})_p\big)
\end{equation}
for all $X\in L(\tilde{\mathcal{A}})$. Noticing that $\mathcal{A}_p$ is an arrangement in $\mathbb{F}_p^d$,  then from \eqref{finite-convolutionformula}, we arrive at
\begin{equation}\label{finite-convolutionformula1}
\sum_{X\in L(\tilde{\mathcal{A}_p})}\chi(\tilde{\mathcal{A}_p}/X,p)\chi(X,p)=p^d(p-1)\chi(\mathcal{A}_p,p).
\end{equation}
Together with \eqref{cong2} and (ii) of \autoref{characteristic-field}, we obtain
\[
\chi(\tilde{\mathcal{A}}/X,p)=\chi\big((\tilde{\mathcal{A}}/X)_p,p\big) ,\,\chi(\mathcal{A}+H_{\bm\alpha_X,a_X},p)=\chi\big((\mathcal{A}+H_{\bm\alpha_X,a_X})_p,p\big),\,\chi(\mathcal{A},p)=\chi(\mathcal{A}_p,p).
\]
By using the above relations to \eqref{finite-convolutionformula1}, we obtain that there exist infinitely many prime numbers $p$ satisfying
\begin{equation}\label{Prime}
\sum_{X\in L(\tilde{\mathcal{A}})}\chi(\tilde{\mathcal{A}}/X,p)\chi(X,p)=p^d(p-1)\chi(\mathcal{A},p).
\end{equation}
According to the definition \eqref{Characteristic-Polynomial} of characteristic polynomial, clearly $\sum_{ X\in L(\tilde{\mathcal{A}})}\chi(\tilde{\mathcal{A}}/X,t)\chi(X,t)$ and $t^d(t-1)\chi(\mathcal{A},t)$ are polynomials in $t$ of degree at most $2d+1$. Thus these two polynomials agree on infinitely many prime numbers via \eqref{Prime}. Then we have $\sum_{X\in L(\tilde{\mathcal{A}})}\chi(\tilde{\mathcal{A}}/X,p)\chi(X,p)=t^d(t-1)\chi(\mathcal{A},t)$ from the fundamental theorem of algebra. We finish the proof.
\end{proof}

\section{Restriction of hyperplane arrangements}\label{SEC4}
As an application of the one-element extension of hyperplane arrangement, we shall further investigate the restriction to hyperplane of a hyperplane arrangement. For any hyperplane $H$ of $\mathbb{F}^d$, the restriction $\mathcal{A}/H$ of $\mathcal{A}$ on $H$ is a hyperplane arrangement in $H$ defined as
\[
\mathcal{A}/H:=\{H'\cap H\ne\emptyset\mid H'\in \mathcal{A},H\nsubseteq H'\}.
\]
In particular, when $\mathbb{F}^d$ is a Euclidean space, Zaslavsky's formula \cite{Zaslavsky1975} shows
\[
r(\mathcal{A}/H)=(-1)^{d-1}\chi(\mathcal{A}/H, -1).
\]
Most recently, Fu-Wang \cite{Fu-Wang2021} studied the restrictions of a linear arrangement, which gives the classification of those restrictions and establishes the order-preserving relations of their Whitney numbers of both kinds and region numbers. In this section, we will extend Fu-Wang's result to general hyperplane arrangements and classify the restrictions $\mathcal{A}/H$ for all hyperplanes $H$ of $\mathbb{F}^d$.

Recall that $O$ is the space spanned by the normal vectors ${\bm\alpha}_1,{\bm\alpha}_2,\ldots,{\bm\alpha}_m$ of the hyperplanes in $\mathcal{A}$. Clearly the restriction $\mathcal{A}/O=\{H\cap O\ne\emptyset\mid H\in\mathcal{A}, O\nsubseteq H\}$ is an essential arrangement in $O$, and
\[
L(\mathcal{A})\cong L(\mathcal{A}/O)\quad\And\quad\chi(\mathcal{A},t)=t^{d-\dim(O)}\chi(\mathcal{A}/O,t).
\]
So we may assume that $\mathcal{A}$ is essential. Then \autoref{classification-1} directly gives rise to the following result, which classifies $L(\mathcal{A}/H)$ and the characteristic polynomials $\chi(\mathcal{A}/H,t)$ for all hyperplanes $H$ in $\mathbb{F}^d$  via the induced adjoint arrangement $\tilde{\mathcal{A}}$ defined in \eqref{extension-adjoint-arrangement}.
\begin{corollary}\label{contraction}
Suppose $\mathcal{A}$ is essential. Let $H:\bm\alpha\cdot\bm x=a$ and $H':\bm\alpha'\cdot\bm x=a'$.  If $X_{(\bm\alpha,a)}\subseteq X_{(\bm\alpha',a')}$ for $X_{(\bm\alpha,a)}, X_{(\bm\alpha',a')}\in L(\tilde{\mathcal{A}})$, then
\[
L(\mathcal{A}/H)\cong L(\mathcal{A}/H')\quad\And\quad \chi(\mathcal{A}/H,t)=\chi(\mathcal{A}/H',t).
\]
\end{corollary}
\begin{proof}
Recall from \autoref{classification-1} that we have $L(\mathcal{A}+H)\cong L(\mathcal{A}+H')$. It directly induces the isomorphism $L(\mathcal{A}+H/H)\cong L(\mathcal{A}+H'/H')$. So we obtain $L(\mathcal{A}/H)=L(\mathcal{A}+H/H)\cong L(\mathcal{A}+H'/H')=L(\mathcal{A}/H')$. From the definition of the characteristic polynomial, we get $\chi(\mathcal{A}/H,t)=\chi(\mathcal{A}/H',t)$ immediately.
\end{proof}
Below we will establish the order-preserving relations of the Whitney numbers of both kinds and region numbers of restrictions $\mathcal{A}/H$ for all hyperplanes of $\mathbb{F}^d$ via the geometric lattice $L(\tilde{\mathcal{A}})$.
\begin{corollary}\label{restriction1}
Suppose $\mathcal{A}$ is essential. Let $H:\bm\alpha\cdot\bm x=a$ and $H':\bm\alpha'\cdot\bm x=a'$.  If $X_{(\bm\alpha,a)}\subseteq X_{(\bm\alpha',a')}$ for $X_{(\bm\alpha,a)}, X_{(\bm\alpha',a')}\in L(\tilde{\mathcal{A}})$, then for $i=0,1,\ldots,d-1$, we have
\[
 w_i^+(\mathcal{A}/H)\le w_i^+(\mathcal{A}/H'),\quad W_i(\mathcal{A}/H)\le W_i(\mathcal{A}/H')\quad\And\quad r(\mathcal{A}/H)\le r(\mathcal{A}/H').
\]
\end{corollary}
\begin{proof}
Let $(\bm\alpha_{m+1},a_{m+1})=(\bm\alpha,a)$ and $(\bm\alpha'_{m+1},a'_{m+1})=(\bm\alpha',a')$.  Note that for any hyperplane $H$ in $\mathbb{F}^d$, we have $\mathcal{A}+H/H=\mathcal{A}/H$. It follows from \autoref{NBC-lemma4} that
\[
\NBC_i(\mathcal{A}/H_{\bm\alpha_{m+1},a_{m+1}})\subseteq\NBC_i(\mathcal{A}/H_{\bm\alpha'_{m+1},a'_{m+1}}), \quad i=0,1,\ldots,d-1.
\]
Applying \autoref{affine-NBC}, we obtain $w_i(\mathcal{A}/H_{\bm\alpha_{m+1},a_{m+1}})\le w_i(\mathcal{A}/H_{\bm\alpha'_{m+1},a'_{m+1}})$.
Together with Zaslavsky's formula, we can arrive at
\[r(\mathcal{A}/H_{\bm\alpha_{m+1},a_{m+1}})=\sum_{i=0}^{d-1} w_i^+(\mathcal{A}/H_{\bm\alpha_{m+1},a_{m+1}})\le \sum_{i=0}^{d-1} w_i^+(\mathcal{A}/H_{\bm\alpha'_{m+1},a'_{m+1}})=r(\mathcal{A}/H_{\bm\alpha'_{m+1},a'_{m+1}}).\]
Next we will show that $ W_i(\mathcal{A}/H_{\bm\alpha_{m+1},a_{m+1}})\le W_i(\mathcal{A}/H_{\bm\alpha'_{m+1},a'_{m+1}})$. Recall from \autoref{NBC-lemma1} that we have
\[
\rank_{ \mathcal{A}+H_{\bm\alpha_{m+1},a_{m+1}}}\big(J\cup\{m+1\}\big)\le \rank_{ \mathcal{A}+H_{\bm\alpha'_{m+1},a'_{m+1}}}\big(J\cup\{m+1\}\big),\quad J\subseteq[m].
\]
From \autoref{NBC-lemma3}, we have
\begin{equation}\label{Rank-4}
\rank_{ \mathcal{A}/H_{\bm\alpha_{m+1},a_{m+1}}}(J)\le \rank_{ \mathcal{A}/H_{\bm\alpha'_{m+1},a'_{m+1}}}(J),\quad J\subseteq[m].
\end{equation}
Next we will establish the rank-preserving injection $\theta: L(\mathcal{A}/H_{\bm\alpha_{m+1},a_{m+1}})\to L(\mathcal{A}/H_{\bm\alpha'_{m+1},a'_{m+1}})$.  For any $Y\in L(\mathcal{A}/H_{\bm\alpha_{m+1},a_{m+1}})$, let  $I_Y\subseteq [m]$ such that $\rank_{ \mathcal{A}/H_{\bm\alpha_{m+1},a_{m+1}}}(I_{  Y})=\#I_{  Y}$  and $Y=\bigcap_{i\in I_{  Y}}H_{\bm\alpha_i,a_i}\bigcap H_{\bm\alpha_{m+1},a_{m+1}}$. Define $\theta(Y)=\bigcap_{i\in I_{  Y}}H_{\bm\alpha_i,a_i}\bigcap H_{\bm\alpha'_{m+1},a'_{m+1}}$. It follows from \eqref{Rank-4} that $\#I_{  Y}\ge \rank_{ \mathcal{A}/H_{\bm\alpha'_{m+1},a'_{m+1}}}(I_{  Y})\ge \rank_{ \mathcal{A}/H_{\bm\alpha_{m+1},a_{m+1}}}(I_{  Y})\ge \#I_{  Y}$, which means
\[
\rank_{ \mathcal{A}/H_{\bm\alpha'_{m+1},a'_{m+1}}}(I_{  Y})=\rank_{ \mathcal{A}/H_{\bm\alpha_{m+1},a_{m+1}}}(I_{  Y})=\#I_{  Y}.
\]
This implies $\bigcap_{i\in I_{  Y}}H_{\bm\alpha_i,a_i}\bigcap H_{\bm\alpha'_{m+1},a'_{m+1}}\ne\emptyset$ and $\theta(Y)\in L(\mathcal{A}/H_{\bm\alpha'_{m+1},a'_{m+1}})$, i.e., $\theta$ is well-defined. Moreover, it also indicates that $\theta$ is a rank-preserving map. To verify the injectivity of $\theta$, suppose $\theta(Y_1)=\theta(Y_2)=Y'\in L(\mathcal{A}/H_{\bm\alpha'_{m+1},a'_{m+1}})$ and $I_{  Y_1},I_{  Y_2}\subseteq [m]$ satisfying that
\[
Y_j=\bigcap_{i\in I_{  Y_j}}H_{\bm\alpha_i,a_i}\bigcap H_{\bm\alpha_{m+1},a_{m+1}}\quad \And\quad
\rank_{ \mathcal{A}/H_{\bm\alpha_{m+1},a_{m+1}}}(I_{  Y_j})=\#I_{  Y_j},\quad j=1,2.
\]
Clearly $Y'=\bigcap_{i\in I_{  Y_1}}H_{\bm\alpha_i,a_i}\bigcap H_{\bm\alpha'_{m+1},a'_{m+1}}=\bigcap_{i\in I_{  Y_2}}H_{\bm\alpha_i,a_i}\bigcap H_{\bm\alpha'_{m+1},a'_{m+1}}$. This means that $\bigcap_{i\in I_{Y_1}\cup I_{Y_2}}H_{\bm\alpha_i,a_i}\in L(\mathcal{A})$ and for all $k\in I_{  Y_2}\setminus I_{  Y_1}$,
\[
Y'=\bigcap_{i\in I_{  Y_1}\cup I_{  Y_2}}H_{\bm\alpha_i,a_i}\bigcap H_{\bm\alpha'_{m+1},a'_{m+1}}=\bigcap_{i\in I_{  Y_1}\sqcup\{k\}}H_{\bm\alpha_i,a_i}\bigcap H_{\bm\alpha'_{m+1},a'_{m+1}}.
\]
It yields
\[
\rank_{ \mathcal{A}/H_{\bm\alpha'_{m+1},a'_{m+1}}}(I_{  Y_1})=\rank_{ \mathcal{A}/H_{\bm\alpha'_{m+1},a'_{m+1}}}\big(I_{  Y_1}\cup\{k\}\big),\quad k\in I_{  Y_2}\setminus I_{  Y_1}.
\]
We claim $Y_1=Y_2$. Otherwise, we have $\bigcap_{i\in I_{Y_1}}H_{\bm\alpha_i,a_i}\ne\bigcap_{i\in I_{Y_2}}H_{\bm\alpha_i,a_i}$. It implies that there is $k_1\in I_{  Y_2}\setminus I_{  Y_1}$ such that $\bigcap_{i\in I_{Y_1}\sqcup\{k_1\}}H_{\bm\alpha_i,a_i}\subsetneqq\bigcap_{i\in I_{Y_1}}H_{\bm\alpha_i,a_i}$. Note that $Y'\subseteq \bigcap_{i\in I_{Y_1}\sqcup\{k_1\}}H_{\bm\alpha_i,a_i}$ and $\dim(\bigcap_{i\in I_{Y_1}}H_{\bm\alpha_i,a_i})-\dim(Y')\in\{0,1\}$. So we have $Y'=\bigcap_{i\in I_{Y_1}\sqcup\{k_1\}}H_{\bm\alpha_i,a_i}\in L(\mathcal{A})$. It follows from \autoref{lemma0} that $\bigcap_{i\in I_{  Y_1}\cup I_{Y_2}}H_{\bm\alpha_i,a_i}\bigcap H_{\bm\alpha_{m+1},a_{m+1}}\bigcap L_0(\mathcal{A})=Y'\cap L_0(\mathcal{A})\ne\emptyset$ since $\mathcal{A}$ is essential. Namely, $\bigcap_{i\in I_{  Y_1}\cup I_{Y_2}}H_{\bm\alpha_i,a_i}\bigcap H_{\bm\alpha_{m+1},a_{m+1}}\ne\emptyset$. Together with $Y_1\ne Y_2$, we have that there is $k_2\in I_{  Y_2}\setminus I_{  Y_1}$ such that
\[
\rank_{ \mathcal{A}/H_{\bm\alpha_{m+1},a_{m+1}}}\big(I_{  Y_1}\cup \{k_2\}\big)>\rank_{ \mathcal{A}/H_{\bm\alpha_{m+1},a_{m+1}}}(I_{  Y_1}).
\]
It follows from $\rank_{ \mathcal{A}/H_{\bm\alpha_{m+1},a_{m+1}}}(I_{  Y_1})=\rank_{ \mathcal{A}/H_{\bm\alpha'_{m+1},a'_{m+1}}}(I_{  Y_1})$ that $\rank_{ \mathcal{A}/H_{\bm\alpha_{m+1},a_{m+1}}}\big(I_{  Y_1}\cup \{k_2\}\big)>\rank_{ \mathcal{A}/H_{\bm\alpha'_{m+1},a'_{m+1}}}\big(I_{  Y_1}\cup \{k_2\}\big)$, contradicting \eqref{Rank-4}. Hence, we obtain $Y_1=Y_2$, i.e., $\theta$ is an injection. The injectivity of $\theta$ directly gives rise to $ W_i(\mathcal{A}/H_{\bm\alpha_{m+1},a_{m+1}})\le W_i(\mathcal{A}/H_{\bm\alpha'_{m+1},a'_{m+1}})$, which completes the proof.
\end{proof}

\end{document}